%% file: Main.tex
\documentclass[reqno]{amsart}

\usepackage{amssymb}
\usepackage{amsthm}
\usepackage{amsmath}
\usepackage{amsfonts}
\usepackage{comment}
\usepackage{stackengine}
\stackMath
\usepackage{graphicx}
\usepackage{xfrac}
\usepackage{enumerate}
\usepackage{cmll}
\usepackage{MnSymbol}
\usepackage[mathscr]{euscript}

\usepackage{soul}

\usepackage{todonotes}
\usepackage{xargs}

\newcommandx{\change}[2][1=]{\todo[linecolor=blue,backgroundcolor=blue!25,bordercolor=blue,#1]{#2}}

\newcommand{\res}{\mathord{\upharpoonright}} 
\newcommand{\qq}{\mathbb{Q}}
\newcommand{\rr}{\mathbb{R}}
\newcommand{\ff}{\mathbb{F}}
\newcommand{\nn}{\mathbb{N}}
\newcommand{\uu}{\mathbb{U}}
\newcommand{\zz}{\mathbb{Z}}

\newcommand{\cc}{\mathbb{C}}

\newcommand{\trdeg}{\mathrm{trdeg}}

\newcommand{\la}{\langle}
\newcommand{\ra}{\rangle}
\newcommand{\elsub}{\prec}

\newcommand{\mr}{\mathrm{r}}
\newcommand{\md}{\mathrm{d}}
\newcommand{\gr}{\mathrm{gr}}
\newcommand{\gd}{\mathrm{gd}}
\newcommand{\mmr}{\mathrm{mr}}
\newcommand{\mmd}{\mathrm{md}}
\newcommand{\crk}{\mathrm{cr}}
\newcommand{\cdg}{\mathrm{cd}}
\newcommand{\mcl}{\mathrm{mcl}}
\newcommand{\acl}{\mathrm{acl}}
\newcommand{\mtp}{\mathrm{mtp}}
\newcommand{\atp}{\mathrm{atp}}

\newcommand{\chara}{\mathrm{char}}
\newcommand{\subsim}{\abovebaseline[-2.6pt]{\ \stackanchor[0.3pt]{\subset} {\sim }\  }}
\newcommand{\ndeq}{\abovebaseline[-.5pt]{\ \stackanchor[2pt]{\scriptscriptstyle{\mathrm{nd}}} {=}\  }}
\newcommand{\ACF}{\mathrm{ACF}}
\newcommand{\RN}[1]{\textup{\uppercase\expandafter{\romannumeral#1}}}

\newcommand{\vartu}{\mathrm{var}^{\RN{2}, \overset{ \overset{\scriptstyle{\text{\bf \bf \textvisiblespace}   }}   {\ }} {\ }   }_L }

\newcommand{\vart}[1]{\mathrm{var}^{\RN{2}, #1}_L }

\makeatletter
\newsavebox\myboxA
\newsavebox\myboxB
\newlength\mylenA

\newcommand*\xbar[2][0.75]{%
    \sbox{\myboxA}{$\m@th#2$}%
    \setbox\myboxB\null
    \ht\myboxB=\ht\myboxA%
    \dp\myboxB=\dp\myboxA%
    \wd\myboxB=#1\wd\myboxA
    \sbox\myboxB{$\m@th\overline{\copy\myboxB}$}
    \setlength\mylenA{\the\wd\myboxA}
    \addtolength\mylenA{-\the\wd\myboxB}%
    \ifdim\wd\myboxB<\wd\myboxA%
       \rlap{\hskip 0.5\mylenA\usebox\myboxB}{\usebox\myboxA}%
    \else
        \hskip -0.5\mylenA\rlap{\usebox\myboxA}{\hskip 0.5\mylenA\usebox\myboxB}%
    \fi}
\makeatother

\newcommand{\TN}{\mathrm{ACFC}}

\newcommand{\TNp}{\mathrm{ACFC}_p}

\newcommand{\fpa}{\mathbb{F}_{p}^{ac} } 
\newcommand{\qa}{\mathbb{Q}^{ac}}				
\newcommand{\Lprc}{ \mathit{L}_{prc} } 

\newtheorem{thm}{Theorem}[section]
\newtheorem{lem}[thm]{Lemma}

\newtheorem*{thm*}{Theorem}
\newtheorem*{conj*}{Conjecture}
\newtheorem*{prop*}{Proposition}
\newtheorem*{fact*}{Fact}
\newtheorem{prop}[thm]{Proposition}

\newtheorem{cor}[thm]{Corollary}
\newtheorem*{rem*}{Remark}
\newtheorem*{eg*}{Example}

\begin{document}
\sloppy

\title[Algebraically Closed Fields with a Character]{Algebraically Closed Fields with a Generic Multiplicative Character}
\author{Tigran Hakobyan, Minh Chieu Tran}

\address{Department of Mathematics\\
University of Illinois at Urbana-Cham\-paign\\
Urbana, IL 61801\\
U.S.A.}
\email{hakobya2@illinois.edu}

\address{Department of Mathematics, University of Illinois at Urbana-
Champaign, Urbana, IL 61801, U.S.A.}
\curraddr{}
\email{mctran2@illinois.edu}
\date{\today}
\maketitle

\input{0.Abstract}

\input{1.Introduction}

\input{2.Axiomatization}

\input{3.ClassificationAndCompleteness}

\input{4.Substructures}

\input{5.DefinableSets}
\input{6.GenericRankAndDegree}

\input{7.OmegaStabilityAndMorleyRank}

\bibliographystyle{amsplain}
\bibliography{8.Research}

\end{document}

%% file: 0.Abstract.tex
\begin{abstract}

\noindent We study the model theory of the $2$-sorted structure $(\mathbb{F}, \mathbb{C};\chi)$, where $\mathbb{F}$ is an algebraic closure of a finite field of characteristic $p$, $\mathbb{C}$ is the field of complex numbers and $\chi: \mathbb{F} \to \mathbb{C}$ is an injective, multiplication preserving map. We obtain an axiomatization $\mathrm{ACFC}_p$ of $\mathrm{Th}(\mathbb{F},\mathbb{C};\chi)$ in a suitable language $L$, classify the models of $\mathrm{ACFC}_p$ up to isomorphism, prove a modified model companion result, give various descriptions of definable sets inside a model of $\mathrm{ACFC}_p$, and deduce that $\mathrm{ACFC}_p$ is $\omega$-stable and has definability of Morley rank in families.

\end{abstract}

%% file: 1.Introduction.tex
\section{Introduction}

\noindent 
Fields with characters occur in many places; see for example 
Kowalski~\cite{Kowalski} for a case where also definability plays a role.
This suggested to us to
look for model-theoretically tame pairs of fields with character maps 
between them.  

Throughout, $(F, K; \chi)$ is a structure where $F$ and $K$ are integral 
domains (usually fields), and $\chi: F \to K$ satisfies 
$\chi(ab)=\chi(a)\chi(b)$ for all $a,b\in F $, $\chi(0)=0 \text{ and } \chi(1)=1.$
Then $(F, K; \chi)$ is naturally a structure in  the $2$-sorted language $L$ which consists of two disjoint copies of the language of rings, augmented by a unary function symbol $\chi$. We call $\chi$ with the above properties a \textbf{character}.

We are particularly interested in the cases where $F$ is an algebraic closure $\ff$ of a finite field, $K$ is the field $\cc$ of complex numbers and $\chi: \ff \to \cc$ is {\em injective}. 
From now on, we let $(\ff, \cc; \chi)$ range over structures with these properties.
Corollary~\ref{UniquenessOfChar} below says 
that for each prime $p$, there is up to $L$-isomorphism exactly one  $(\ff, \cc; \chi)$
such that $\chara(\ff) = p$. 
In this paper we show that the 
$L$-theory $\text{Th}(\ff, \cc; \chi)$ is tame in various ways. For precise statements we
need some more terminology. 

Let $(F,K;\chi)$ be given. For a tuple $\alpha=(\alpha_1,\dots, \alpha_n)\in F^n$, $n \in \nn^{\geq 1}$, and $k=(k_1,\dots,k_n)\in \nn^n$ we set $\alpha^k:=\alpha_1^{k_1}\cdots\alpha_n^{k_n}$. We call $\alpha$ {\bf multiplicatively dependent}
if $\alpha^k=\alpha^l$ for some distinct $k,l\in \nn^n$, and {\bf multiplicatively independent} otherwise.
We say that \( \chi: F \to K\) is { \bf generic} if it is injective and 
for all multiplicatively independent $\alpha \in F^n$, $n \in \nn^{\geq 1}$,
the tuple $\chi(\alpha):=\big(\chi(\alpha_1),\dots, \chi(\alpha_n)\big) \in K^n$ is algebraically independent in the fraction field of $K$ over its prime field.

\begin{thm}
There is a recursive set $\TN$ of $\forall\exists$-axioms in $L$ such that:
\begin{enumerate}
\item for all $(F, K;\chi)$, $(F, K;\chi)\models \TN$ if and only if $F$ and $K$ are algebraically closed fields, $\chara(K)=0$ and $\chi:F \to K$ is generic;
\item for all $p$ prime, if $\chara(\ff)=p$, then $(\ff, \cc; \chi)\models \TN$.
\end{enumerate}
\end{thm}

\noindent If $p$ is either prime or zero, let $\TN_p$ be the set of $\forall\exists$-axioms in $L$ obtained from $\TN$ by adding the statements expressing  $\chara(F)=p$ where $(F, K;\chi)$ is an $L$-structure.

\noindent Let $\kappa, \lambda$ be (possibly finite) cardinals. In Section 3 we prove the following classification result. (If $p=0$, set $\ff_p:=\qq$.) 

\begin{thm} \label{1.2}
For any $p,\kappa$ and $\lambda$, there is up to isomorphism a unique model $(F, K; \chi)$ of $\TN_p$ such that $\trdeg(F \mid \ff_p) = \kappa, \trdeg\big(K \mid \qq( \chi(F))\big)= \lambda$. 
\end{thm}

\noindent
By the wealth of results by Shelah in \cite{Shelah}, we can get the following:

\begin{cor}
$\TN_p$ is superstable, shallow, without the dop, without the otop, without the fcp.
\end{cor}

\noindent
By an analogue of Vaught Test, we have:

\begin{thm}
$\TN_p$ is complete.
\end{thm}

\noindent
In section 4 we characterize the substructures of models of $\TN_p$:  
\begin{prop}\label{sub} Given $(F, K; \chi)$, the following are equivalent:
\begin{enumerate}
\item $(F, K; \chi)$ is a substructure of a model of $\TNp$;
\item$\chi$ is generic and $\chara(F)=p$.
\end{enumerate}
\end{prop}

\noindent
When is a substructure of a model of  $\TNp$ an 
elementary submodel?  It is not enough that the substructure is 
a model of $\TNp$:

\begin{prop}
$\TN_p$ is not model complete.
\end{prop}

\noindent
To deal with the above question we define
a {\bf regular submodel} of a model $(F', K';\chi')$ of $\TN_p$ to be
a substructure $(F, K;\chi)\models \TN_p$ of $(F', K';\chi')$ such that
 $\qq\big(\chi(F')\big) $ is linearly disjoint with $ K$ over $\qq\big(\chi(F)\big)$ in $K'$. 
The more complicated notion of {\bf regular $L$-substructure} will be defined in Section 4. Below we fix a set  $\TNp(\forall)$ of universal 
$L$-sentences whose models are the substructures of models of $\TNp$, as in
Proposition~\ref{sub}. 

\begin{thm}
\(\TNp \) is the {\rm regular} model companion of \( \TNp(\forall) \). That is:
\begin{enumerate}
\item for models of $\TNp$, the notions of {\rm regular submodel} and 
{\rm elementary submodel} are equivalent;
\item every model of \( \TNp(\forall) \) is a regular substructure of a model of \( \TNp \).
\end{enumerate}
\end{thm}

\noindent In Section $5$, we show that every definable set in a fixed model $(F, K; \chi)$ of $\TN_p$ has a simple description. This is comparable to the fact that every definable set in a model of $\ACF$ is a boolean combination of algebraic sets. A set $S \subseteq K^n$ is {\bf algebraically presentable} if 
$$S\  =\ \bigcup_{\alpha \in D}V_\alpha$$
for some definable $D\subseteq F^m$ and
definable family $\{ V_\alpha\}_{\alpha \in D}$ of $K$-algebraic subsets of $K^n$. 
 Algebraically presentable sets should be thought of as geometrically simple. They are also
existentially definable of a particular form. We also define in Section~$5$ the related notion of {\bf $0$-algebraically presentable} sets. The main result is:

\begin{thm}
If $X \subseteq K^n$ is definable, then $X$ is a boolean combination of algebraically presentable subsets of $K^n$. Furthermore, if $X$ is $0$-definable, $X$ is a boolean combination of 0-algebraically presentable subsets of $K^n$.
\end{thm}

\noindent
The general case where a definable set is not a subset of $K^n$ can be easily reduced to the above special case.

\newpage
\noindent We have better results for definable subsets of $F^m$:

\begin{thm}
Let $D\subseteq F^m$. If $D$ is definable, then $D$ is definable in the field $F$. If $D$ is $0$-definable, then $D$ is $0$-definable in the field $F$. 
Suppose $D = \chi^{-1}(V)$ with  $V\subseteq K^m$ a $K$-algebraic set. Then $D$ is $F$-algebraic. If $V$ is defined over $\qq$ in the field sense, then $D$ is defined over $\ff_p$ in the field sense. 
\end{thm}

\noindent
Still working in a fixed model $(F, K; \chi)$ of $\TN_p$, we show in Section $6$ that every definable set has another description which is comparable to the fact that every definable set in a model of $\ACF$ is a finite union of quasi-affine varieties. A special case of the above description is when $X = \bigcup_{\alpha \in D} V_\alpha$ where $\{ V_\alpha\}_{\alpha \in D}$ is a definable family of varieties over $K$ such that 
\begin{enumerate}
\item for some $k\in \nn$, $D$ is a subset of $F^k$ definable in the language of rings,
\item $V_\alpha$  and $V_\beta$ are disjoint for distinct $ \alpha, \beta \in D $.
\end{enumerate}

\noindent We can show that not every definable set $X \subseteq K^n$ has a description as in the above special case.
However, a description approximating the above picture can be obtained. This in particular allows us to define a {\bf geometric rank} $\gr$  and a {\bf geometric degree} $\gd$ on the definable sets. 

In Section $7$, we show that the geometric rank and the geometric degree defined in the previous section coincide with the model-theoretic notions of Morley rank and Morley degree.
Using this, it is easy to deduce that the theory is $\omega$-stable. We then study some behaviors of these notions of ranks.

\begin{prop}
If $X \subseteq K^n$ and $ X'\subseteq K^{n'}$ are definable, $\gr(X) = \omega\cdot \rho_{{}_K} + \rho_{{}_F}$ and $  \gr(X') = \omega\cdot \rho'_{{}_K} + \rho'_{{}_F}$ with $\rho_{{}_K}, \rho_{{}_F}, \rho'_{{}_K}, \rho'_{{}_F} \in \nn$,  then  $\gr(X\times X') = \omega\cdot (\rho_{{}_K}+\rho'_{{}_K}) +\rho_{{}_F}+\rho'_{{}_F}$.
\end{prop}

\begin{thm}
Suppose $ \{ X_b\}_{b \in Y}$ is a definable family of subsets of $K^n$. Then for each ordinal \(\rho\), the set \( \big\{ b \in Y :  \gr(X_b)= \rho \big\} \) is definable.
\end{thm}

\noindent
The technique we developed might also provide a step towards proving that $\TN_p$ has the
{\em definable  multiplicity} property. 
We use the geometric understanding of Morley rank and Morley degree to classify strongly minimal sets up to non-orthogonality:

\begin{prop}
For any strongly minimal set $X \subseteq K^n$ there is a finite-to-one definable map from $X$ to $F$.
\end{prop}

\noindent
The structure $(\ff, \cc; \chi)$ is similar to various known structures, 
for example $(\cc, \qa)$ where $\qa$ is the set of algebraic numbers regarded as an additional unary relation on $\cc$.
The study of the latter stretches back to Robinson (see ~\cite{Robinson}). Analogues of some of our results for $(\cc, \qa)$ seem to be known as folklore; see
for example \cite{Louominpair}. However, our structure is mathematically even more closely related to $( \cc, \uu)$ where $\uu \subseteq \cc$ is the group of all roots of unity regarded as an additional unary relation.
In fact, we can almost view $(\ff, \cc; \chi)$ as $( \cc, \uu)$ with some 
extra relations on $\uu$.
In consequence, several results of this paper are either directly implied or easy adaptations of results in \cite{Zilber} and \cite{DriesGun}.
These include axiomatization, $\omega$-stability, quantifier reduction; the corresponding result of $(\kappa, \lambda)$-transcendental categoricity is known, according to Pillay,  but not written down anywhere.
There are also several results that hold in the above mentioned two structures and ought to have suitable analogues in our structure but we have not proven them yet.
These include the study of imaginaries and definable groups;  see \cite{Pillay} and \cite{Haydar}.

\noindent On the other hand, some of our results are new, which also yield more information on the structures $(\cc, \qa)$ and $( \cc, \uu)$ as well.
Through the notion of genericity, we obtain a more conceptual characterization of the class of models of $\TN_p$ other than using the axioms.  From \cite{Zilber},
we can already see that every model of $\TN_p$  satisfies the properties of this characterization. In this paper, we show the reverse direction.
It is clear that one can obtain a characterization of the class of models of $\text{Th}( \cc, \uu)$  in the same way.
Even though both use Mann's theorem in an essential way, our axiomatization strategy is slightly different from the strategy used in \cite{Zilber} and \cite{DriesGun}. 
This modification, in particular, allows us to also axiomatize the class of substructures of the models and achieve the regular model companion result mentioned above. 
The regular model companion result should have analogues for $( \cc, \qa)$ and $(\cc, \uu)$ as well. To our knowledge, the method in Sections $6$ and $7$ and the result about definability of Morley Rank in families is not known.
We expect that this method  can be applied to $( \cc, \qa)$ and $(\cc, \uu)$ and generalized further to study the appropriate notions of dimension and multiplicity in other types of pairs.

A natural continuation of our project is to study the expansion $(\ff, \cc; \chi, \rr)$ of $(\ff, \cc; \chi)$ where $\rr\subseteq \cc$ is the set of real numbers.  At the definability level, this amounts to also including the metric structure on $\cc$ into the picture. Towards this end, the second author has considered a reduct of $(\ff, \cc; \chi, \rr)$; see \cite{Minh} for details.

%% file: 2.Axiomatization.tex
\section{Axiomatization}

\noindent Throughout, let $m, n$ ranges over the set of natural numbers (which includes zero), $p$ be either a prime number or zero, $t =(t_1, \ldots, t_n)$, $u =(u_1, \ldots, u_m)$ be tuples of variables of the first sort and $x= (x_1, \ldots, x_n)$, $y = (y_1, \ldots, y_m)$ be tuples of variables of the second sort. If $a$ is in $X^n$, then $a = (a_1, \ldots, a_n)$ with $a_i \in X$ for $i \in \{1, \ldots n\}$. If $A \subseteq K^\neq := K \setminus \{ 0 \}$, set $\la A \ra_K $ to be the set of elements in $K^\neq$ which are in the subgroup generated by $A$ in the fraction field of $K$; when the context is clear, we will write $\la A \ra $ instead of $\la A \ra_K $. We note that in the preceding definition, $\la A \ra $  is a submonoid of $K^\neq$ and is a subgroup of $K^\times$ if $K$ is a field. If $P_1, \ldots, P_m$ are systems of polynomial in $K[x]$, we let $Z(P_1, \ldots, P_m) \subseteq K^n$ be the set of their common zeros.

In this section, we also assume that $A, B \subseteq K^\neq$ and $C \subseteq K$. Let \( \acl_C(A) \) denote the elements of $K$ satisfying a nontrivial polynomial equation with coefficients in $\zz[A,C]$. We will give a definition of the notion of genericity which is slightly more general than what was given in the introduction.
This is necessary  for the purpose of axiomatization and will also play an important role in the next two sections.


The \textbf{multiplicative closure} of $A$ \textbf{over} $B$, denoted by $\mcl_B(A)$, is the set 
$$ \big\{ a \in K^\neq : a^n \in \la A \cup B \ra \text{ for some } n \big\}. $$ 
%
%
We note that if $K$ is a field, the notion of multiplicative closure over $B$ coincides with the notion of divisible closure over $B$, viewing \(K^\times\) as a \( \zz\)-module. 
We say \(A\) is {\bf multiplicatively independent over} \(B\) if 
$$ a \notin \mcl_B \big(A \backslash \{a\}\big) \text{ for all } a \in A. $$
A {\bf multiplicative basis} of \( A\) \textbf{over} \(B\) is an \(A'\subseteq A\) such that $A'$ is multiplicatively independent over \(B\) and \( A \subseteq \mcl_B(A') \).
General facts about pregeometry give us that
there is a multiplicative basis of \(A\) over \(B\); furthermore, any two such bases have the same cardinality.
When $B =\emptyset$, we omit the phrase {\it over $B$} in the definition and the subscript $B$ in the notation. We also note that $\mcl( \emptyset)= \{ 1\} $.
\newpage

\noindent We say \( A \) is {\bf generic} if for all multiplicatively independent $a \in \la A \ra^n$ we also have $a$ is algebraically independent.
We say $A$ is \(C\)-\textbf{generic over} \(B\) if  for all $B$-multiplicatively independent $a \in \la A \ra^n$ we also have $a$ is algebraically independent over $B \cup C$. The following follows easily from the exchange property of $\text{mcl}$:


\begin{lem} \label{GenericRobust}
Suppose $A$ is $C$-generic over $B$. Then the preceding statement continues to hold as we:
\begin{enumerate}
\item replace $K$ by an integral domain $K'$ such that $A, B \subseteq {K'}^\times$ and $C \subseteq K'$,
\item replace $B$ with $B' \subseteq K^\neq$ such that $\mcl(B') = \mcl(B)$,
\item replace $C$ with $C' \subseteq K$ such that $\acl_B(C') =\acl_B(C)$,
\item replace $A$ with $A'\subseteq K^\neq$ such that $\mcl_B(A') = \mcl_B(A)$.
\end{enumerate}
\end{lem}

\begin{cor} \label{Equiv1}
The following equivalence holds: \(A\) is \(C\)-generic over \(B\) if and only if there is a family \( A' \subseteq A\) such that $A'$ is algebraically independent over $B\cup C$ and  \( A \subseteq \mcl_B(A')  \).
\end{cor}

%
\noindent The notions of {\it multiplicative closure} and {\it multiplicative independence} can be understood using polynomials. A {\bf monomial} in $x$ is an element of $\qq[x]$ of the form  $x^k$ with $k \in \nn^n$. Likewise, a {\bf $B$-monomial} in $x$ is an element of $K[x]$ the form $b^l x^k$ with $b\in B^m$, $l \in \nn^m$ and $k \in \nn^n$. In this section $M$ and $N$ are $B$-monomials. 
A {\bf $B$-binomial} is a polynomial of the form $M-N$. If, moreover, $M$ and $N$ are monomials, we call $M-N$ a {\bf binomial}.
We call a $B$-binomial $M-N$ {\bf nontrivial}  if 
$$M\ =\ b^{l_M}x^{k_M}\ \text{ and }\ N\ =\ b^{l_N}x^{k_N}\ \text{ for some }\ b \in B^m\ \text{ and distinct }\ k_M, k_N \in \nn^n.$$ 
It is easy to see that for $a \in K^\neq$, $a$ is in $\mcl_B(A)$ if and only if $a$ is a zero of a non-trivial $(A \cup B)$-binomial of one variable. 
Then $A$ is multiplicatively independent over $B$ if for all $n$, for all $a \in A^n$, $a$ is not in the zero-set of a nontrivial $B$-binomial of $n$ variables.


Suppose $K$ is a field, $H \subseteq G \subseteq K^\times $ are groups, $C$ is a subfield of $K$, $g \in G^n$, and $a \in K^n$.
The  \textbf{multiplicative type} of $g$ \textbf{over} \(H\), denoted by \( \mtp_H(g)\), is the quantifier free type of \(g\) in the language of groups with parameters from \(H\).
We can easily see that \( \mtp_H(g)\) is completely characterized by the $H$-binomials vanishing on \(g\). 
If \(H=\{1\}\), we simply call this the  \textbf{multiplicative type} of $g$, and denote this as \( \mtp(g)\).
Likewise, the  \textbf{algebraic type} of $a$ \textbf{over} \(C\), denoted by \( \atp_C(a)\), is the quantifier free type of \(a\) in the language of rings with parameters from \(C\).
Then \( \atp_C(a)\) is completely characterized by the polynomials in $C[x]$ vanishing on \(a\).
If \(C=\qq\), we call this the  \textbf{algebraic type} of $a$, and denote this by \( \atp(a)\). Suppose \( c\) is an $n$-tuple of elements in \( K\) and $d$ is an element in $K$. 
A solution \(a\) of the equation \(c\cdot x =d \) is called {\bf non-degenerate} if we have \(c_{i_1}a_{i_1}+\cdots+c_{i_m}a_{i_m} \neq 0 \) for all \( \{i_1, \ldots, i_m  \} \subsetneq
 \{1, \ldots, n\} \).

\begin{prop}\label{Equiv2}
 Suppose $K$ is a field, $H \subseteq G \subseteq K^\times $ are groups, and $C$ is a subfield of $K$. Moreover, suppose $\mcl(H)\cap G = H$. The following are equivalent:
\begin{enumerate}

\item $G$ is $C$-generic over \(H\);

\item for all $g, g' \in G^n$, if \( \mtp_H(g)= \mtp_H(g')\) then \( \atp_{C(H)}(g)= \atp_{C(H)}(g') \);

\item for all $ g \in G^n$ and all $P\in C(H)[x]$, $P$ vanishes on $g$ if and only if $P$ is in the ideal $I_g$ of $C(H)[x]$ generated by $H$-binomials vanishing on $g$;

\item if $c\in C^n$, and \(g \in G^n\) is a non-degenerate solution of the equation \( c\cdot x =1 \), then \(g\) is in \(H^n\).

\end{enumerate}
Without the condition  $\mcl(H)\cap G = H$, we still have \((4) \Rightarrow (3) \Rightarrow (2) \Rightarrow (1)\).
\end{prop}

\begin{proof}

Throughout the proof, we suppose $K, C, G$ and $H$ are as given. We first show that $(4)$ implies $(3)$.
Suppose \((4)\) holds, and \(P\) is in \(C(H)[x] \) such that \(P(g)=0\). 
For our purpose, we can arrange that
$$P\ =\ \sum_{i=1}^{k} c_iM_i- M_{k+1}\ \text{ where } c_i \text{ is in } C^\times \text{ and } M_i \text{ are } H\text{-monomials for } i \in \{1, \ldots, k\}.$$ 
The cases where $k=0, 1$ are immediate. Using induction, suppose $k>1$ is the least case the statement has not been proven.
Then \( \big( M_1(g), \ldots, M_{k+1}(g)\big)\) is a non-degenerate solution of \( c_1y_1+ \cdots+ c_{k}y_{k} - y_{k+1} =0 \).
Hence, $M_{k+1}(g) \neq 0$ and 
$$ \left( M_1(g)M_{k+1}^{-1}(g), \ldots, M_{k}(g)M_{k+1}^{-1}(g) \right) $$
is a non-degenerate solution of $c_1y_1+ \cdots+c_{k}y_{k} = 1 $.
Hence, it follows from (4) that  $M_i(g) M_{k+1}^{-1}(g) = h_i \in H$ for $i \in \{ 1,  \ldots, k\}$.
As a consequence,  $$( c_1h_1+ \cdots+ c_{k}h_{k} - 1) M_{k+1} \ =\ P - \sum_{i=1}^{k}c_i(M_i-h_iM_{k+1}) $$ vanishes on $g$. 
As $M_{k+1}(g) \neq 0$, the above implies $c_1h_1+ \cdots+ c_{k-1}h_{k-1} - 1 =0$. Thus 
$P = \sum_{i=1}^{k}c_i(M_i-h_iM_{k+1})$ which is in $I_g$. The conclusion follows.

To show that $(3)$ implies $(2)$, let $g$ and $g'$ be as in $(2)$.
Then a $H$-binomial vanishes on \(g\) if and only if it vanishes on \( g'\), and so \( I_g =I_{g'}\). 
The desired conclusion then follows from \( (2) \).


We now show that $(2)$ implies $(1)$.
Suppose we have \((2)\) and $g\in G^n$ is multiplicatively independent over $H$. 
We can arrange that \(K\) is algebraically closed by $(1)$ of Lemma~\ref{GenericRobust}. 
The case where $n=0$ is trivial.
Using induction, suppose $n>0$ is the least case the statement has not been proven. 
Then $g_1, \ldots, g_{n-1}$ are algebraically independent over $C(H)$.
Assume \(P \in C(H)[x]\) is non-trivial. 
As \( g_1, \ldots, g_{n-1} \) are algebraically independent over \(C(H)\), we get that 
$$P(g_1, \ldots, g_{n-1}, x_n)\ \neq\ 0 \ \text{ in } C(H, g_1, \ldots, g_{n-1})[x_n],$$ and so it has at most finitely many roots. 
As a consequence, \( P(g_1, \ldots, g_{n-1}, g_n^m) \ \neq\  0 \) for some $m > 0$. 
Because \(g = (g_1, \ldots, g_n)\) is multiplicatively independent over \(H\), for all $m$, \((g_1, \ldots, g_{n-1},g_n^m)\) has the same multiplicative type over \(H\) as \((g_1, \ldots, g_n)\). By (2), for all $m$, \((g_1, \ldots, g_{n-1}, g_n^m)\) has the same algebraic type over $C(H)$ as \((g_1, \ldots, g_n)\).
Therefore, \(P(g_1, \ldots, g_n) \neq 0\). Since \(P\) is chosen arbitrarily, \(g\) is algebraically independent over \( C(H)\), and so we have \((1) \).


We show that  $(1)$ implies $(4)$.
Suppose we have \( (1)\), \(\mcl(H) \cap G=H \), and $g \in G^n$ is a non-degenerate solution of  \( c\cdot x =1 \).
Let $G'$ be the subgroup of $G$ generated by $g$. 
As $\mcl(H) \cap G=H $, the group $G'\slash (H \cap G')$ is torsion-free of finite rank, and so we can choose $g'_1, \ldots, g'_k$ in $G'$ multiplicatively independent over $H$ such that 
$$ g_i\ =\ M'_i(g'_1, \ldots, g'_k)\ \text{ for some } H\text{-monomial  }M'_i \text{ for } i \in \{1, \ldots, n\}.$$
As $g'_1, \ldots, g'_k$ are multiplicatively independent over $H$, they are algebraically independent over $C(H)$ by (1).
As 
$$g\ =\ \big(M'_1(g'_1, \ldots, g'_k), \ldots, M'_n(g'_1, \ldots, g'_k)\big)$$ is a non-denegerate solution of the equation $c\cdot x =1$, $g'_j$ must appear with power $0$ in all $M'_i$ for all $i \in \{ 1, \ldots, n\}$ and $j \in \{1, \ldots, k\}$.
Hence $g$ is in $H^n$.  

Finally, we observe that the condition \(\mcl(H) \cap G=H \) is only used in showing $(1)$ implies $(4)$. Thus, the other implications still hold without this condition.
\end{proof}

\noindent
Here, we present another property of genericity as a corollary of the previous proposition.

\begin{cor} \label{GenTrans} 
We have the following:
\begin{enumerate}
\item for \(A \subseteq A' \subseteq A'' \subseteq K^ \times\), \( A'\) is \(C\)-generic over \(A\) and \( A''\) is \(C\)-generic over \(A'\) if and only if \(A''\) is \(C\)-generic over \(A\);
\item suppose $\{A_\alpha\}_{\alpha < \kappa} \) is a sequence of subsets of $K^\times$ such that   $A_\alpha \subseteq A_{\alpha+1}$ and \( A_{\alpha+1}\) is \(C\)-generic over \(A_\alpha\) for all $\alpha < \kappa$, and $A_\beta = \bigcup_{\alpha < \beta} A_\alpha$  for all limit ordinals $\beta$. If $A = \bigcup_{\alpha < \kappa} A_\alpha$, then $A$ is $C$-generic over $A_\alpha$ for all $\alpha < \kappa$.
\end{enumerate}

\end{cor}

\begin{proof}
By Lemma~\ref{GenericRobust}, we can arrange that $C$ and $K$ are fields and all the $A_\alpha$'s involved are multiplicatively closed in $K$. In particular, each $A_\alpha$ with the multiplication is a group. The conclusions follow easily from the equivalence of $(1)$ and $(4)$ of Proposition~\ref{Equiv2}.
\end{proof}

\noindent We call a polynomial in $\qq[x]$ {\bf special} if it has the form $\prod_\zeta (M- \zeta N)$ where $\zeta$ ranges over the set of $k$-th primitive roots of unity for some $k>0$ and some monomials $M$ and $N$.

\begin{prop}\label{GenericEquivalence}
Suppose $K$ is a field, $G \subseteq K^\times $ is a group, and $U$ is the set of all roots of unity in $K$. Moreover, suppose \(\chara(K)=0  \). Then the following are equivalent:
\begin{enumerate}
\item G is generic;
\item for all \(g, g' \in G^n\), if \(\mtp(g)= \mtp(g') \) then \(\atp(g)=\atp(g') \); 
\item for all $g \in G^n$ and $P \in \qq[x]$, $P$ vanishes on $g$ if and only if $P$ is in $\sqrt{J_g}$ where $J_g \subseteq  \qq[x]$ is the ideal generated by the special polynomials vanishing on $g$;

\item if $c$ is in $\qq^n$, and $g \in G^n$ is a non-degenerate solution of the equation $ c\cdot x =1 $ then $g$ is in $U^n$.

\end{enumerate}
\end{prop}

\begin{proof}
Throughout the proof, we suppose $K, G$ and $U$ are as stated.
We first prove that $ (1)$ implies $(3)$. As the statement is independent of the ambient field, we can arrange that \(K\) is algebraically closed. It is clear even without assuming (1) that the backward implication of (3) holds. 
Now we suppose $(1)$ and prove the forward implication of $(3)$. We reduce the problem to finding finitely many special polynomials \(S_1, \ldots, S_l \) such that 
$$ Z(S_1, \ldots, S_l)\ \subseteq\ Z(P) .$$ 
Indeed, suppose we managed to do so. Then, by the Nullstellensatz, this implies \( P^m \) is in the ideal generated by \( S_1, \ldots, S_l\) in \(K[x]\). Hence \(P^m\) is a \(K\)-linear combination of products \(M_i S_j\) for $i \in \{1, \ldots, k\} $, $j \in \{ 1,  \ldots, l\}$ and each $M_i$ a monomial in $x$. By taking a linear basis of \(K\) over \(\qq\) and taking into account the assumption that \( P\) is in \( \qq(x) \), we get \(P^m\) is a \(\qq\)-linear combination of products of \( M_iS_j\) as above. Therefore, \(P\) is in \(\sqrt{J_g}\).

By equivalence of $(1)$ and $(3)$ in Proposition~\ref{Equiv2}, we have that \( P\) lies in the ideal \(I_g\) of \(\qq(U)[x]\) generated by polynomials of the form \( M - \zeta N \) vanishing on $g$ with \(M, N\) monomials in \(x\) and \( \zeta\) a root of unity.
As \(\qq(U)[x]\) is Noetherian, there are binomials \( M_1 - \zeta_1 N_1, \ldots,  M_l - \zeta_l N_l\) generating \(I_g\).
Hence, 
$$ Z(M_1 - \zeta_1 N_1, \ldots,  M_l - \zeta_l N_l)\ \subseteq\ Z(P). $$ Let $\zeta$ be a generator of the subgroup of $U$ generated by $\zeta_1, \ldots, \zeta_l$.
Then there are  natural numbers $s_1, \ldots, s_l$ and $t_1, \ldots, t_l$ such that  $\zeta = \zeta_1^{s_1}\ldots \zeta_l^{s_l}$ and $\zeta_i = \zeta^{t_i}$ for all $i \in \{1, \ldots l \}$. 
Let $$M'\ =\ \prod_{i=1}^l(M_i)^{s_i}\ \text{ and }\ N'\ =\ \prod_{i=1}^l(N_i)^{s_i}.$$
We note that $Z(M_1 - \zeta_1 N_1, \ldots,  M_l - \zeta_l N_l)$ is equal to $$ Z \big(M' - \zeta N', (N')^{t_1}M_1- (M')^{t_1}N_1, \ldots, (N')^{t_l}M_l- (M')^{t_l}N_l\big).$$
Therefore, we might as well assume $P$ vanishes on the zero set of polynomials $M_1(x) - \zeta N_1(x)$, $M_2(x)-N_2(x), \ldots, M_l(x) - N_l(x)$.

With \(\zeta, M_i, N_i \) as in the preceding statement, let \( \zeta\) be a primitive \(k\)-th root of unity. Set 
$$ S_1\ =\ \prod_\varepsilon (M_1- \varepsilon N_1 )\ \text{ where } \varepsilon \text{ ranges over the primitive } k\text{-th roots of unity }$$
and \( S_2 = M_2(x)-N_2(x), \ldots, S_l = M_l(x)-N_l(x) \). 
Note that each \(S_i\) is special. 
Suppose,  \(a \in K^n \) is in the zero set of the ideal of $\qq[x]$ generated by \( S_1, \ldots, S_l\). 
Then there is a primitive \(k\)-th root of unity \( \varepsilon\) such that \( M_1(a)- \varepsilon N_1(a)=0 \). Since \(\chara(K)=0\), there is an automorphism \(\sigma\) of \(K\) such that \( \sigma( \varepsilon) = \zeta \). 
Hence, $$ M_1\big(\sigma(a)\big)- \zeta N_1\big(\sigma(a)\big) \ =\ S_2\big( \sigma(a)\big)\ =\ \ldots\ =\ S_l\big( \sigma(a) \big)\ =\ 0.$$
By the choice of \(\zeta, M_i, N_i \), we have \( P\big( \sigma(a) \big) =0\). As \(P \) is in \(\qq[x]\), \(P(a) =0\). Thus, we have proven the reduction and hence (3).

Next, we prove that $(3)$ implies $(2)$. Suppose \( (3) \), and \(g, g'\) have the same multiplicative type.
Let $ S $ be a special polynomial such that 
$S(g)=0$ and $S = \prod_\zeta (M- \zeta N)$ where $\zeta$ ranges over all the primitive $k$-th roots of unity for some $k>0$.
Then $M(g)N^{-1}(g)$ is a primitive $k$-th $k$-th root of unity, so $M^k-N^k$ vanishes on $g$ but $M^l-N^l$ does not vanish on $g$ for $0<l<k$.
As $g, g'$ have the same multiplicative type, $$M^k(g')-N^k(g')\ =\ 0\ \text{ but }\ M^l(g')-N^l(g')\ \neq\ 0\ \text{ for }0<l<k.$$
So $M(g')N^{-1}(g')$ is a primitive $k$-th root of unity and $S(g')=0$. 
Hence \( J_g = J_{g'} \), and so \(\atp(g)=\atp(g') \). Thus, we have \( (2)\). 

The argument for (2) implying (1)  is the same as the argument for (2) implying (1) in Proposition~\ref{Equiv2}. Finally, by $(2)$ of Lemma~\ref{GenericRobust}, \( G\) is generic if and only if \(G\) is generic over \( G \cap U\). We note that \(\mcl(G\cap U) \cap G = G \cap U \), so the equivalence between \( (1) \) and \( (4) \) follows immediately from the equivalence between \( (1) \) and \( (4) \) in Proposition~\ref{Equiv2}. 
\end{proof}

\begin{prop} \label{Regularity}
Suppose $K$ is a field, $G \subseteq K^\times $ is a group, $g$ is in $G^n$ and $H$ is a subgroup of $G$ such that $G$ is generic over $H$. Moreover, suppose \(\chara(K)=0  \), and $\mcl(H)\cap G = H$. Then $\qq(G)$ is a regular field extension of $\qq(H)$.
\end{prop}

\begin{proof}
As $\chara(K)=0$, $\qq(G)$ is a separable field extension of $\qq(H)$, so it suffices to check that $\qq(H)$ is algebraically closed in $\qq(G)$. 
Suppose $P, Q \in \qq[x]$, and $g \in G^n$ is such that $P(g)Q^{-1}(g)$ is algebraic over $\qq(H)$. Let $G'$ be the subgroup of $G$ generated by $g$. 
As $\mcl(H) \cap G=H $, $G'\slash (H \cap G')$ is torsion-free of finite rank, we can choose $g'_1, \ldots, g'_k$ in $G'$ multiplicatively independent over $H$ such that $$g_i\ =\ M'_i(g'_1, \ldots, g'_k)\ \text{ where }  M'_i \text{ is } H\text{-monomial for }i \in \{1, \ldots, n\}$$
Hence we can find  $P', Q'$ coprime in $\qq(H)[y_1, \ldots, y_k]$ such that $P'(g')Q'^{-1}(g')$ is equal to  $P(g)Q^{-1}(g)$.
As $g'_1, \ldots, g'_k$ are multiplicatively independent over $H$, they are algebraically independent over $\qq(H)$. 
Therefore, in order to have $P'(g')Q'^{-1}(g')$ algebraic over $\qq(H)$, the polynomials $P', Q'$ must have degree $0$ and  so $P'(g')Q'^{-1}(g')$ is in $\qq(H)$. The conclusion follows.
\end{proof}

\noindent We recall the following version of a theorem of Mann from \cite{Mann}:

\begin{thm*}[Mann] 
Let $U$ be the group of roots of unity in $\qa$. There is a recursive function \( d: \nn \to \nn \) such that if \(a_1, \ldots, a_n\) are in \( \qq\)  and \((y_1, \ldots, y_n)\) in \(U^n\) is a tuple of non-degenerate solution of the equation $a_1y_1+\cdots+a_ny_n=1$, then \( y_i^{d(n)}=1 \) for all \(i\).
\end{thm*}

\noindent
For an $L$-structure $(F, K; \chi)$, it is easy to see that $\chi$ in generic if and only if $\chi(F^\times)$ is generic  in the sense of this section. As a consequence we have:

\begin{prop} \label{GenericAxiom}
There is a recursive set of universal statements in $L$ whose models $(F, K, \chi)$ are precisely the $L$-structures with $\chi$ generic.
\end{prop}

\begin{proof}
Suppose $F'$ and $K'$ are respectively the fraction fields of $F$ and $K$.
Using only the conditions that $\chi$ is multiplication preserving, $\chi(0) =0$ and $\chi$ is injective, we can extend $\chi$ to an injective character $\chi': F' \to K'$; moreover, $\chi'$ maps multiplicatively independent elements to algebraically independent elements if and only if $\chi$ does so by Lemma~\ref{Equiv1}.
We also note that $(F', K'; \chi')$ is interpretable in $(F, K; \chi)$ in the obvious way.
Hence we can reduce the problem to the case where  $F$ and $K$ are fields.

 Combining  the equivalence between \((1)\) and \((4) \) of Proposition~\ref{GenericEquivalence} and Mann's theorem, \(\chi\) is generic if and only if for all $n$ and all non-degenerate solutions of $a_1x_1+\cdots+a_nx_n=1$ in \(\big(\chi(F)^\times\big)^n\) with $a$ in \(\qq^n\), we have \( x_i^{d(n)} =1 \) for \(i \in \{ 1, \ldots, n\}\).
It is clear that being a non-degenerate solution is definable by a quantifier-free formula.
So we have the desired universal axiom scheme.
\end{proof}

\begin{thm}
There is a recursive set $\TN$ of $\forall\exists$-axioms in $L$ such that:
\begin{enumerate}
\item for all $(F, K;\chi)$, $(F, K;\chi)\models \TN$ if and only if $F$ and $K$ are algebraically closed fields, $\chara(K)=0$ and $\chi:F \to K$ is generic;
\item for $p>0$, $(\fpa, \qa; \chi)\models \TN$.
\end{enumerate}
\end{thm}

\begin{proof}
It follows easily from proposition~\ref{GenericAxiom} that we have the desired axiomatization. 
When $F= \fpa, K = \qa$ and $\chi: F \to K$ is injective, we note that $\chi$ is automatically generic because there is no multiplicative independence between elements of $\chi(F^\times)$.
\end{proof}

\noindent Let $\mathscr{Q}$ be the set of prime powers. For each $q \in \mathscr{Q} $, let $\chi_q: \ff_q \to \qa$ be an injective map with $\chi_q(0)=0$ and $\chi_q(ab)=\chi_q(a)\chi_q(b)$ for all $a,b\in \ff_q$. With exactly the same method we get: 

\begin{prop}
There is a recursive set of $\forall\exists$-axioms $T$ in $L$ with the following properties:
\begin{enumerate}
\item for all $(F, K;\chi)$, $(F, K;\chi)\models T$ if and only if $K$ is an algebraically closed fields with $\chara(K)=0$, $F$ is a pseudo-finite field and $\chi:F \to K$ is generic;
\item  if $\mathscr{U}$ is a non-principal ultrafilter on $\mathscr{Q}$, then $\left( \prod_{q \in \mathscr{Q}}(\ff_q, \qa; \chi_q) \right) \slash \mathscr{U} \models T$.
\end{enumerate}
\end{prop}

\noindent
This also allows us to conjecture that for every $T$-model $(F,K;\chi)$, there is an ultrafilter $\mathscr{U}$ on $\mathscr{Q}$ such that  $ (F,K;\chi) \equiv \left( \prod_{q \in \mathscr{Q}}(\ff_q, \qa; \chi_q) \right) \slash \mathscr{U}$.

%% file: 3.ClassificationAndCompleteness.tex
\section{Classification, completeness and decidability}

\noindent
We keep the notation conventions in the first paragraph of the preceding section and moreover assume in this section that $(F, K; \chi) \models \TN_p$. For a field $K$, we let $K^{\text{ac}}$ denote an algebraic closure of $K$.
We classify the models of $\TN_p$ up to isomorphism. From this we deduce that the theory $\TN_p$ is complete and decidable.

\begin{prop}\label{CharacterIsomorphism}
Suppose $(F, K; \chi_1)$ and $(F, K, \chi_2)$ are models of $\TN_p$ with \( \qq\big(\chi_1(F) \big)= \qq\big(\chi_2(F) \big)\). Then there is an automorphism  \(\sigma\) of \(K\) with \( \chi_2 = \sigma \circ \chi_1 \). 
\end{prop}

\begin{proof}
Suppose $F, K,\chi_1$ and $\chi_2$ are as stated. Let \(\alpha= (\alpha_i)_{i \in I}\) be a listing of the elements of \(F^\times\). As \( \chi_1, \chi_2 \) are group homomorphisms, \( \text{mtp}\big(\chi_1(\alpha)\big) = \text{mtp}\big(\chi_2(\alpha)\big)  \). By Proposition~\ref{GenericEquivalence}, \( \text{atp}\big(\chi_1(\alpha)\big) = \text{atp}\big(\chi_2(\alpha)\big) \),  and so there is a field automorphism 
$$\sigma: \qq\big(\chi_1(F) \big) \to \qq\big(\chi_2(F) \big) $$ 
such that \( \chi_2 = \sigma \circ \chi_1 \).
We can further extend $\sigma$ to a field automorphism of \(\qq\big(\chi_1(F) \big)^{ac}= \qq\big(\chi_2(F) \big)^{ac}  \) and then to an automorphism of \(K\).
\end{proof}

\begin{cor} \label{UniquenessOfChar}
If $p$ is prime, $F = \fpa$ and  $K = \qa$, then there is a unique injective character from $F$ to $K$ up to isomorphism.
\end{cor}

\begin{cor} \label{AutExt}
If $\chi: F \to K$ is generic and if $\sigma$ is an automorphism of $F$, then $\sigma$ can be extended to an automorphism of \((F, K;\chi)\).
\end{cor}

\noindent
We say  $(F,K;\chi) \models \TNp$ is {\bf $(\kappa,\lambda)$-transcendental} if $\trdeg(F \mid \ff_p)=\kappa $ and $\trdeg\big(K \mid \qq(G)\big)  = \lambda $ with $G=\chi(F^\times)$.

\begin{thm}\label{ThmIso}

For any $p$,$\kappa$ and $\lambda$, there is a unique $(\kappa,\lambda)$-transcendental model of $\TNp$ up to isomorphism .
\end{thm}

\begin{proof}

We first prove the uniqueness part of the lemma. Suppose $(F_1,K_1;\chi_1)$ and $(F_2,K_2;\chi_2)$ are $(\kappa,\lambda)$-transcendental models of $\TNp$. Let $G_1$ be $\chi (F_1^\times)$ and $G_2$ be $\chi (F_2^\times)$. As $F_1$ and $F_2$ are algebraically closed of same characteristic and  $\trdeg(F_1 \mid\ff_p) = \trdeg(F_2 \mid \ff_p) $, there is an isomorphism $$\sigma: F_1 \to F_2.$$ Using Proposition~\ref{GenericEquivalence} in a similar way as in the proof of Proposition~\ref{CharacterIsomorphism}, $\sigma$ induces an isomorphism between $ \qq(G_1)$ and $\qq(G_2)$; we will also call this $\sigma$. Finally, since $\trdeg\big(K_1\mid\qq(G_1)\big)$ is equal to $\trdeg\big(K_2\mid \qq(G_2)\big)$ we can extend $\sigma$ to a field isomorphism from $K_1$ to $K_2$. It is easy to check that this is an isomorphism of $L$-structures.

We next prove the existence part of the lemma. For $p>0$, $\TN_p$ clearly has a model. For $p=0$, $\TN_p$ has a model by compactness. We can arrange to have for each $p$ a model $(F, K; \chi)$ of $\TN_p$ such that $|F|, |K| > \max\{\kappa, \lambda, \aleph_0\}$. Choose an algebraically closed subfield $F'$  of $F$ with $\trdeg(F' \mid \ff_p) = \kappa$. Then we have $\trdeg\big(K \mid \qq( \chi(F'))\big) > \lambda$. Choose an algebraically closed subfield $K'$ of $K$ containing $\chi(F')$ with $\trdeg\big(K' \mid \qq(\chi(F'))\big) =\lambda$. We can check that $(F', K'; \chi\upharpoonright_{F'})$ is a $(\kappa,\lambda)$-transcendental model of $\TNp$.
\end{proof}

\begin{cor}
$\TN_p$ is superstable, shallow, without the dop, without the otop, without the fcp.
\end{cor}

\begin{proof}
The first four properties follow from Shelah's main gap theorem \cite[XII.6.1]{Shelah}. The last property follows from  \cite[VII.3.4]{Shelah}.
\end{proof}

\noindent
Next we prove an analog of upward L\"owenheim-Skolem theorem.
\begin{lem}
For $\chi: F \to K$ generic, $K$ is an infinite extension of $\qq\big(\chi(F)\big)$.
\end{lem}

\begin{proof}
Suppose $F, K$ and $\chi$ are as stated. Let $G= \chi( F)$. By Proposition~\ref{Regularity}, if $U$ consists of the roots of unity in $G$, then $\qq(G)$ is a regular extension of $\qq(U)$. Hence, 
$$\big[\qa:\qq(U)\big]\ \leq\ \big[\qa\qq(G): \qq(G) \big].$$ By Galois theory, $\big[\qa:\qq(U)\big] = \infty$. Therefore, $\big[\qa\qq(G): \qq(G)\big]= \infty$ and so $\big[K: \qq(G)\big]= \infty$.
\end{proof}

\begin{lem} \label{KappaExt}
Every model $(F, K; \chi)$ of $\TNp$  has a $(\kappa ,\kappa)$-transcendental elementary extension $(F', K'; \chi')$ for any cardinal $\kappa \geq \max(|F|, |K|)$.
\end{lem}

\begin{proof}
Let $(F, K; \chi)$ and $\kappa$ be as stated.
We construct an elementary extension $(F'', K''; \chi'')$ of $(F, K; \chi)$ with $\text{trdeg}(F''\mid \ff_p) \geq \kappa$ and $\text{trdeg}(K''\mid G'') \geq \kappa$ with $G'' = \chi({F''}^\times)$.
For the later two conditions to hold, it suffices to ensure there are $$\alpha\ \in\ (F'')^\kappa\ \text{ and }\ a\ \in\ (K'')^\kappa$$ such that components of $\alpha$ are all distinct and the components of $a$ are algebraically independent over $G''$. 
Using compactness, we can reduce the problem to verifying the following: for arbitrary $k,m,n$, $w$  of length $m$, $x$ of length $n$ and arbitrary polynomials $P_1, \ldots, P_l $ in $\qq[w,x]$, there are $\alpha$ in $F^k$ and $a$ in $K^n$ such that components of $\alpha$ are pairwise different, and
$$P_i\big(\chi(\beta), a\big)\ \neq\ 0\ \text{ for all } \beta \in F^m \text{ and } i \in \{1 \ldots l\}.$$
It is easy to find $\alpha$ with the desired property. By preceding lemma, $ \big[K: \qq(G)\big]$ is infinite, so we can choose  $a$ so that $\big[\qq(G, a_1, \ldots, a_j): \qq(G, a_1 \ldots,a_{j-1})\big]>N$ for $j \in \{ 1, \ldots, n\}$ where $N$ is the maximum degree of $P_i$ for $ i \in \{1 \ldots l\} $.
We see that this choice of $a$ works. We then get the desired $(F', K'; \chi')$ from $(F'', K''; \chi'')$ by taking the Skolem Hull of the suitable elements.
\end{proof}

\begin{thm}\label{CompAxiom}
For all $p$, $\TNp$ is complete and  decidable. When $p>0$, $\TNp$ axiomatizes $\mathrm{Th}(\ff, \cc; \chi)$ where $\chara(\ff) = p$.
\end{thm}

\begin{proof}
We first show that any two arbitrary models \( (F_1, K_1; \chi_1) \) and \( (F_2, K_2; \chi_2) \) of $\TNp$ are elementarily equivalent. By the preceding lemma, we can arrange that \( (F_1, K_1; \chi_1) \) and \( (F_2, K_2; \chi_2)\)  are both $(\kappa, \kappa)$-transcendental. It follows from  Theorem ~\ref{ThmIso} that for all $p$, $\TNp$ is complete. The remaining conclusions are immediate.
\end{proof}

\begin{cor}
Let $\tau$ be an $L$-statement. The following are equivalent:
\begin{enumerate}
\item $\tau$ is true in some model of $\TN_0$;
\item there are arbitrarily large primes $p$ such that $\tau$ is true in some model of $\TNp$;
\item there is a number $m$ such that for all primes $p>m$, $\tau$ is true in all models of $\TNp$.
\end{enumerate}
\end{cor}

%% file: 4.Substructures.tex
\section{Substructures and elementary substructures}
\noindent
From now on, let $k, l$ range over the set of natural numbers, $s=(s_1, \ldots, s_k)$, $v =(v_1, \ldots, v_l)$ be tuples of variables of the first sort and $w= (w_1, \ldots, w_k)$, $z = (z_1, \ldots, z_l)$ be tuples of variables of the second sort. We also implicitly assume similar conventions for these letters with additional decorations.

\noindent In addition to the notation conventions in the first paragraph of section 2, we assume in this section that $(F, K; \chi)$ has $\chara(K)=0$. We use $\subseteq$ and $\preccurlyeq$ to denote the $L$-substructure and elementary $L$-substructure relations respectively. 
We will characterize the substructures and elementary substructures of a model of $\TN_p$.

\begin{prop} \label{Substructures}
We have $(F, K; \chi)$ is an $L$-substructure of an $\TN_p$-model if and only if $\chi$ is generic and $\chara(F)=p$.
\end{prop}

\begin{proof}
The forward implication is clear. For the other direction, suppose $\chi$ is generic and $\chara(F)=p$. We can embed $(F, K; \chi)$ into an $L$-structure $(F'', K'';\chi'')$ where $F'', K'' $ are respectively the fraction fields of $F, K$ and $\chi''$ is the natural extension of $\chi$ to $F''$. We note that $\chi''$ is still generic. Therefore, we can arrange that $F$ and $K$ are already fields.

Let $G$ be $ \chi(F^\times)$, $F'$ be the algebraic closure of $F$, and  $K'$ be an algebraically closed field containing $K$  such that $\text{trdeg}(K' \mid K)>|F'|$. Let $\{\alpha_i \}_{ i < \kappa}$ be a multiplicative basis of ${F'}^\times$ over $F^\times$. As $\text{trdeg}(K'\mid K)>|F'|$, we can define a map $$\chi': \{\alpha_i \}_{ i < \kappa} \to K'$$ such that the image is algebraically independent over $K$.
Since $\chara(K)=0$, we have $\mcl_G\big(\{ \chi'(\alpha_i) \}_{ i < \kappa}\big)$ in ${K'}^\times$ is divisible. Hence we can extend $\chi'$ to an injective map $\chi': {F'}^\times \to {K'}^\times$ extending $\chi$. Let $G' = \chi({F'}^\times)$.
Then $G'$ is $K$-generic over $G$ by Corollary~\ref{Equiv1}.
Since $G $ is generic, $G'$ is also generic by Corollary~\ref{GenTrans}. Thus the structure $(F', K'; \chi')$ is the desired model of $\TNp$.
\end{proof}

\noindent
Let $(F', K';\chi')$ be an $L$-structure. We say that $(F, K; \chi)$ is a {\bf regular substructure} of $(F', K';\chi')$, denoted as  $(F, K; \chi) \sqsubseteq (F', K';\chi')$,  if  $(F, K; \chi) \subseteq (F', K';\chi')$ and $\chi'( {F'}^\times)$ is $K$-generic over $\chi( F^\times)$. With the use of Proposition~\ref{Equiv2}, it can be seen that the above proof also gives us the following stronger statement:

\begin{cor}\label{Substructures2}
If $\chi$  is generic then there is a model $(F', K'; \chi')$ of $\TNp$ such that 
$(F, K; \chi) \sqsubseteq (F', K';\chi')$.
\end{cor}

\noindent
We now characterize the regular substructure relation for models of $\TN_p$.

\begin{prop}\label{GenericSubstructure}
Suppose $(F,K;\chi) \subseteq (F',K';\chi')$ are models of $\TN_p$. Let $G= \chi(F^\times)$ and $G' = \chi'({F'}^\times)$. Then the following are equivalent:

\begin{enumerate}
\item $(F, K; \chi) \sqsubseteq (F', K';\chi')$;

\item for all $n$, all $P_1, \ldots, P_n \in \qq[w]$ and all $a_1, \ldots, a_n \in K$, if there is a tuple $g' \in {G'}^{k}$ with $P_1(g'), \ldots, P_n(g')$ not all $0$ and  $a_1P_1(g') + \cdots+ a_nP_n(g') =0$, then we can find such a tuple in $G^{k}$;

\item $\qq(G')$ and $K$ are linearly disjoint over $\qq(G)$ in $K'$.
\end{enumerate}

\end{prop}

\begin{proof}

Towards showing that $(1)$ implies $(2)$, suppose (1). Fix $n$, polynomials $P_1, \ldots, P_n \in \qq[w]$, $K$-elements $a_1, \ldots, a_n$ and $g'$ as in $(2)$. 
We want to find $g \in G^{k}$ with $P_1(g), \ldots, P_n(g)$ not all $0$ and  
$$a_1P_1(g) + \cdots+ a_nP_n(g)\ =\ 0.$$
Replacing $(F,K;\chi)$ and $(F',K';\chi')$ concurrently with elementary extensions and noting that $G'$ remains $K$-generic over $G$ by the equivalence between $(1)$ and $(4)$ of Proposition~\ref{Equiv2}, we can arrange that $(F,K;\chi)$ is $\aleph_0$-saturated. By the equivalence between $(1)$ and $(2)$ of Proposition~\ref{GenericEquivalence}, 
$$ \text { if }\ \text{mtp}(g)\ =\ \text{mtp}(g')\ \text{ then }\ P_1(g), \ldots, P_n(g)\ \text{ are not all }0.$$ 
By the equivalence between (1) and (3) of Proposition~\ref{Equiv2}, there are $G$-binomials $M_1-N_1, \ldots, M_l - N_l$ vanishing on $g'$ such that $$M_1(g)-N_1(g)\ =\  \cdots \ =\  M_l(g) - N_l(g) \ =\ 0\ \text{ implies }\ a_1P_1(g) + \cdots+ a_nP_n(g) \ =\ 0.$$
Let $\alpha=\chi^{-1}(g)$, $\alpha'=\chi^{-1}(g')$ and $\chi^{-1}M_i, \chi^{-1}N_i$ be the pullbacks of $M_i$ and $N_i$ under $\chi$ for $i \in  \{1, \ldots,l\}$. 
It suffices to find $\alpha \in (F^\times)^{k}$ with $\mtp(\alpha)=\mtp(\alpha')$ and $$\chi^{-1}M_1(\alpha)-\chi^{-1}N_1(\alpha)\ =\  \cdots\ =\ \chi^{-1}M_l(\alpha)-\chi^{-1}N_l(\alpha)\ =\ 0.$$ Such $\alpha$ can be found as $F$ is an elementary substructure of $F'$ in the language of field and $F$ is $\aleph_0$-saturated. Thus we have (2).

It is immediate that  (2) implies (3). Towards showing that (3) implies (1), suppose (3) and $g' \in (G')^n$ is algebraically dependent over $K(G) =K$. We need to show that $g'$ is multiplicatively dependent over $G$. Pick a non-trivial $P \in K[x]$ with $P(g')=0$. Choose a linear basis $(b_i)_{i\in I}$ of $K$ over $\qq(G)$. Then 
$$P\ =\ \sum_{i \in I} P_ib_i\  \text{ with }\ P_i \in \qq(G)[x]\ \text{ for } i \in I $$ 
and $P_i = 0$ for all but finitely many $i \in I$. 
Hence $\sum_{i \in I} P_i(g')b_i =0 $. By (3), $(b_i)_{i\in I}$ remains linearly independent over $\qq(G')$. Therefore, $\sum_{i \in I} P_i(g')b_i =0 $ implies that $P_i(g')=0$ for all $i \in I$. 
Since $P$ is non-trivial, at least one $P_i$ is non-trivial, and hence $g'$ is algebraically dependent over $\qq(G)$. Now, $G'$ is generic so $G'$ is generic over $G$ by (1) of Corollary~\ref{GenTrans}. By the definition of genericity, $G'$ is $\qq(G)$-generic over $G$.  Hence, $g'$ is multiplicatively dependent over $G$ which is the desired conclusion.
\end{proof}

\begin{cor} \label{ElSub1}
For $(F, K; \chi), (F', K';\chi') \models \TN_p$, if  $(F,K;\chi) \preccurlyeq (F',K';\chi')$, then $(F,K;\chi)\sqsubseteq (F',K';\chi')$.
\end{cor}

\begin{proof}
This follows from the equivalence between (1) and (2) of Proposition~\ref{GenericSubstructure}.
\end{proof}

\begin{cor}
For all $p$, $\TN_p$ is not model complete, and has no model companion in $L$. The same conclusion applies to $\TN$.
\end{cor}

\begin{proof}
We show that $\TNp$ is not model complete. Let $(F, K; \chi)$ and $(F', K'; \chi)$  be models of $\TNp$ such that the former is an $L$-substructure of the latter and the latter is $\kappa$-saturated with $\kappa > |F| + |K|$. Set $G = \chi(F^\times)$ and $G'= \chi({F'}^\times) $. We get by saturation $a, b \in G'$ algebraically independent over $K$.
We will show that 
$$(F, K''; \chi)\ \not \sqsubseteq\ (F', K'; \chi)\ \text{ where }\ K''\ =\ K (a + b) ^{ac},$$ which yields the desired conclusion by Corollary~\ref{ElSub1}.
Fix $g'\in G' \cap K''$. Then $g'$ and $a + b$ are algebraically dependent over $K$ and therefore so are $g'$, $a$ and $b$. 
Suppose $G'$ is $K''$-generic over $G $. As a consequence, $g'$, $a$ and $b$ are also multiplicatively dependent over $G$.  By replacing $g'$ with some power of it if needed, we arrange $g' = M(a, b)$ for some $G$-monomial $M$. Then $M(a, b)$ is in $K'' = K(a, b)^{ac}$ and so $a$ and $b$  are algebraically dependent over $K$, a contradiction. As a consequence, $(F, K''; \chi) \not \sqsubseteq (F', K'; \chi)$.

\noindent Suppose $\TNp$ has a model companion $T$ in $L$.
Take any model $\mathscr{M}$ of $T$. 
Then $\mathscr{M}$ is an $L$-substructure of $(F,K;\chi) \models \TN_p$ which itself is an $L$-substructure of $\mathscr{N} \models T$.
Let $\varphi(x)$ be an existential formula in $L$ such that $\TN_p \models \forall x \varphi(x)$. Hence, for all $a$ with components in $\mathscr{M}$ of suitable sorts, $(F,K;\chi) \models \varphi(a)$ and so $\mathscr{N} \models \varphi(a)$. As $T$ is model complete, we also have  $\mathscr{M} \models \varphi(a)$. 
Since $\TN_p$ is  a set of $\forall \exists$-formulas, $\mathscr{M} \models \TN_p$.
On the other hand, $\TNp$ is complete. Hence, $T = \text{Th}(  \mathscr{M}) = \TNp$, a contradiction as $\TN_p$ is not model complete.

It is easy to see that if a theory $T$ has a model companion, then any of its extension also has a model  companion. The final conclusion thus follows.
\end{proof}

\noindent
There are clearly some obstructions for one model of $\TN_p$  to be an elementary submodel of another model of $\TN_p$ that contains it. We will show that these are the only obstructions. For the main theorem of this section we need the following two technical lemmas:

\begin{lem} The following statements hold:
\begin{enumerate}
\item If $(F, K;\chi) \sqsubseteq (F', K';\chi')$ and $ (F', K';\chi') \sqsubseteq (F'', K'';\chi'')$, then we have $(F, K;\chi) \sqsubseteq (F'', K'';\chi'')$.
\item Suppose $ (F_0, K_0;\chi_0) = (F, K; \chi)$ and $ (F_m, K_m;\chi_m) \sqsubseteq (F_{m+1}, K_{m+1};\chi_{m+1})$ for every $m$. If \(F' = \bigcup_m F_m, K' = \bigcup_m K_m \) and $\chi' = \bigcup_m \chi_m $, then we have $(F, K;\chi) \sqsubseteq (F', K';\chi')$.
\end{enumerate}
\end{lem}

\begin{proof}
This follows from Corollary~\ref{GenTrans} and the definition of genericity.
\end{proof}

\begin{lem}
Let $ (F_0, K_0;\chi_0) = (F, K; \chi)$, $ (F_m, K_m;\chi_m) \sqsubseteq (F_{m+1}, K_{m+1};\chi_{m+1})$ for each $m$, and \(F' = \bigcup_m F_m, K' = \bigcup_m K_m \), $\chi' = \bigcup_m \chi_m $. If $(F_m, K_m;\chi_m)$ is a model of $\TN_p$ with $|K_m| =\kappa$ for each $m$ and $(F, K; \chi)$
 is $(\kappa, \kappa)$-transcendental, then $ ( F', K'; \chi')$ is a $(\kappa, \kappa)$-transcendental model of $\TN_p$. 
\end{lem}

\begin{proof}
In addition to the above notations, let $G = \chi( F^\times)$ and $G' = \chi( {F'}^\times)$. As $\TNp$  is a set of $\forall\exists$-formulas,  $(F', K'; \chi')$ is a model of $\TNp$.  
Since $ (F, K; \chi)$ is $(\kappa,\kappa)$-transcendental, there is $a \in K^\kappa$ with all components algebraically independent over $\qq(G)$. 
By the preceding lemma, $$ (F, K; \chi)\ \sqsubseteq\ (F', K'; \chi').$$
By Proposition~\ref{GenericSubstructure}, $K $ and $ \qq(G')$ are linearly disjoint over $ \qq(G)$ in $K'$.
Hence the components of $a$ remain algebraically independent over $\qq(G')$. Therefore, $\trdeg\big(K' \mid  \qq(G')\big) \geq \kappa$.
Also, $\trdeg(F' \mid \ff_p) \geq  \kappa$.
Hence $|F'|=|K'|= \kappa$ by a cardinality argument. 
Thus, $ ( F', K'; \chi')$  must be $(\kappa, \kappa)$-transcendental.
\end{proof}

\begin{thm} \label{ElSub2}
\(\TNp \) is the regular model companion of \( \TNp(\forall) \). That is:
\begin{enumerate}
\item every model of \( \TNp(\forall) \) is a regular substructure of a model of \( \TNp \);
\item when $(F, K; \chi) \sqsubseteq (F',K';\chi')$ are models of $\TN_p$, then we have $(F, K; \chi) \preccurlyeq (F',K';\chi')$.
\end{enumerate}
\end{thm}

\begin{proof}

We have (1) follows from Corollary~\ref{Substructures2}. The proof of (2) requires some preparation. We let $L^{+}$ be the language obtained by adding to $L$ an $n$-ary relation $R_{P_1,\ldots, P_n}$ for each $n$, and each choice of polynomials $P_1, \ldots, P_n \in \qq[w]$. The theory $\TNp^+$ is obtained by adding to $\TN_p$ the following axioms for each choice of $n,P_1, \ldots, P_n$:
$$R_{P_1,\ldots, P_n}(x)\leftrightarrow \exists s \left( \Big( \bigvee_{i=1}^n P_i \big(\chi(s)\big) \neq 0\Big) \wedge \Big( x_1 P_1 \big(\chi(s) \big) + \cdots+ x_n P_n \big(\chi(s) \big) =0 \Big)  \right). $$ 
We note that $\TNp^+$ is still a complete $\forall\exists$-theory. If $(F, K ; \chi)$ is a model of $\TN_p$, we will let $(F, K ; \chi, R)$ be its natural expansion to a model of $\TNp^+$; here, $R$ represents all the possible $R_{P_1,\ldots, P_n}$ for simplicity of notation. Then, by equivalence of $(1)$ and $(2)$ of Proposition~\ref{GenericSubstructure}, (2) of this theorem is equivalent to saying the theory $\TNp^+$ is model complete in $L^+$. 

It suffices to show that all models of $\TNp^+$ are existentially closed. Suppose we have a counterexample $ (F, K; \chi, R)$. We first reduce to the case where $ (F, K; \chi)$ is moreover $(\kappa, \kappa)$-transcendental for some infinite $\kappa$. By assumption, there is an $\TN_p^+$-model $(F'', K''; \chi'', R'')$ extending $ (F, K; \chi, R)$ as a $L^+$-substructure such that the latter is not existentially closed in the former. Consider the structure $(F'', K''; \chi'', R'', F, K, \chi, R)$ in the language where $F, K, R, \chi$ are regarded as relations on $(F'', K''; \chi'', R'')$.
Note that if we replace this structure with an elementary extension we will still have $ (F, K; \chi, R)$ a non-existentially closed $\TN_p^+$-submodel of $ (F'', K''; \chi'', R'')$. Using a similar trick as in Lemma~\ref{KappaExt}, we can add the condition that $ (F, K; \chi)$ is $(\kappa, \kappa)$-transcendental.

Next, we will construct $ (F', K'; \chi', R')$ existentially closed such that $(F',K'; \chi')$ is $(\kappa,\kappa)$-transcendental. We start with $ (F_0, K_0; \chi_0, R_0) = (F, K; \chi, R) $, the structure obtained at the end of the previous paragraph, and for each $m>0$ construct the $\TNp^+$-model $ (F_m, K_m; \chi_m, R_m)$ as follows.
Choose $(F_{m+1}, K_{m+1}; \chi_{m+1}, R_{m+1})$ to be an $\TNp^+$-model extending $(F_m, K_m; \chi_m, R_m)$ realizing a maximal consistent set of existential formulas with parameters from $(F_m, K_m; \chi_m, R_m$); concurrently, we use downward L\"owenheim-Skolem theorem to arrange $|K_m| = \kappa$.  Let $$F'\ =\ \bigcup_m F_m,\ K'\ =\ \bigcup_m K_m,\ \chi'\ =\ \bigcup_m \chi_m,\ R'\ =\ \bigcup_m R_m.$$ By construction, $(F',K'; \chi')$ is an existentially closed model of $\TN_p^+$. By the equivalence between (1) and (2) of Proposition~\ref{GenericSubstructure},  $( F_m, K_m; \chi_m)$ is a regular substructure of $( F_{m+1}, K_{m+1}; \chi_{m+1})$.
It follows from the preceding lemma that $(F', K'; \chi')$ is $(\kappa, \kappa)$-transcendental.

Finally, by Theorem~\ref{ThmIso}, $ (F, K; \chi)$  and  $ (F', K'; \chi')$ are isomorphic. Hence, $(F, K; \chi, R)$ is also isomorphic to $ (F', K'; \chi', R')$, a contradiction to the fact that the former is not existentially closed but the latter is.
\end{proof}

\begin{cor} \label{ElementarilyEmbeddable}
Suppose $(F, K; \chi) \models \TN_p$ is $(\kappa, \lambda)$-transcendental and $(F', K'; \chi') \models \TN_p$ is $(\kappa', \lambda')$-transcendental. Then $(F, K; \chi) $ can be elementarily embedded into $(F', K'; \chi')$ if and only if $\kappa \leq \kappa'$ and    $\lambda \leq \lambda'$.
\end{cor}

\begin{proof}
We prove the forward direction. Suppose $(F, K; \chi)$ and $(F', K'; \chi')$ are as stated and $(F, K; \chi)$ is elementarily embeddable into $(F', K'; \chi')$. We can arrange that $(F, K; \chi) \preccurlyeq (F', K'; \chi') $. Clearly, $\kappa' \geq \kappa$.  Furthermore, by Corollary~\ref{ElSub1} and $1 \Leftrightarrow(3)$ of Proposition~\ref{GenericSubstructure}, $\qq\big(\chi({F'}^\times)\big)$ and $K$ are linearly disjoint over $\qq\big(\chi(F^\times)\big)$ in $K'$, and so  $\lambda' \geq \lambda$.

For the backward direction, using Theorem~\ref{ThmIso} it suffices to show that a fixed $(\kappa, \lambda)$-transcendental model $(F, K; \chi)$ of $ \TN_p$ has a $(\kappa', \lambda')$-transcendental elementary extension. Find $F'$ extending $F$ with $|F'| =\kappa'$, take $K''$ a sufficiently large algebraically closed field containing $K$ and construct $K' \subseteq K''$ in the same fashion as in the proof of Proposition~\ref{Substructures} to obtain $(F', K'; \chi')$ such that $(F,K;\chi) \sqsubseteq (F',K';\chi')$. This is the desired model by the preceding theorem.
\end{proof}

%% file: 5.DefinableSets.tex
\section{Definable sets I}
\noindent
We keep the notation conventions in the first paragraphs of sections 2 and 4. Moreover, we assume in this section that $ (F, K; \chi) \models \TN_p$. A set $X\subseteq K^n$ is {\it definable in  the field $K$} if it is definable in the underlying field $K$. 
In this case, we use $\mr_K(X), \md_K(X)$ to denote the corresponding Morley rank and degree. %
We equip $K^n$ with the Zariski topology on $K$, also referred to as the  {\it $K$-topology}. A  {\it $K$-algebraic set} is a closed set in this topology.
We define the corresponding notions for $F$ in a similar fashion. In this section, we show that definable sets in a model of $\TN_p$ has a geometrically and syntactically simple description. The following observation is immediate:
%

\begin{prop} \label{transferRemark}
Let $\chi: F^k \times K^n \to K^{k+n}: ( \alpha, a) \mapsto \big( \chi(\alpha), a\big) $.
If $X \subseteq F^k \times K^n$  is definable, then $\chi \res_X: X \to \chi(X)$ is a definable bijection. 
Moreover, $X \subseteq K^n$ is definable over $( \gamma, c) \in F^l \times K^m$ if and only if $X$ is definable over  $\big( \chi(\gamma), c\big) \in K^{l+m}$. 
\end{prop}

\noindent
Hence, we restrict our attention to definable subsets of $K^n$. For a similar reason, we only need to consider sets definable over $c \in K^m$.

Suppose $(H_b)_{ b \in Y}$ and $(H'_{b'})_{ b' \in Y'}$ are families of subsets of $K^{n}$. We say $(H_b)_{ b \in Y}$ {\bf contains}  $(H'_{b'})_{ b' \in Y'}$ if for each $b' \in Y'$, there is $b\in Y$ such that $H_b = H'_{b'}$; $(H_b)_{ b \in Y}$ and $(H'_{b'})_{ b' \in Y'}$ are {\bf equivalent} if each contains the other.
A {\bf combination} of $(H_b)_{ b \in Y}$ and $(H'_{b'})_{ b' \in Y'}$ is any family of subsets of $K^n$ containing both $(H_b)_{ b \in Y}$ and $(H'_{b'})_{ b' \in Y'}$, which is minimal with these properties in the obvious sense.
A {\bf fiberwise intersection} of $(H_b)_{ b \in Y}$  and $(H'_{b'})_{ b' \in Y'}$ is any family of subsets of $K^n$ equivalent to $( H_b \cap H'_{b'})_{(b,b') \in Y \times Y'}$. 
A {\bf fiberwise union} of $(H_b)_{ b \in Y}$ and $(H'_{b'})_{ b' \in Y'}$ is any family of subsets of $K^n$ equivalent to $( H_b \cup H'_{b'})_{(b,b') \in Y \times Y'}$. 
A {\bf fiberwise product} of $(H_b)_{ b \in Y}$ and $(H'_{b'})_{ b' \in Y'}$ is any family of subsets of $K^{2n}$ equivalent to $( H_b \times H'_{b'})_{(b,b') \in Y \times Y'}$; this definition can be generalized in an obvious way for two families of subsets of different ambient spaces. The following is immediate from the above definitions:

\begin{lem} \label{fiber}
Suppose $(H_b)_{ b \in Y}$  and $(H'_{b'})_{ b' \in Y'}$ are families of subsets of $K^n$. Let $X= \bigcup_{b \in Y} H_b$ and $X'= \bigcup_{b' \in Y'} H'_{b'}$. Then we have the following:
\begin{enumerate}
\item $X \cup X'$ is the union of any combination of $(H_b)_{ b \in Y}$ and $(H'_{b'})_{ b' \in Y'}$;
\item $X \cap X'$ is the union of any fiberwise intersection of  $(H_b)_{ b \in Y}$ and $(H'_{b'})_{ b' \in Y'}$;
\item $X \cup X'$ is the union of any fiberwise union of $(H_b)_{ b \in Y}$ and $(H'_{b'})_{ b' \in Y'}$;
\item $X \times X'$ is the union of any fiberwise product of $(H_b)_{ b \in Y}$ and $(H'_{b'})_{ b' \in Y'}$. This part of the lemma can be generalized in an obvious way for two families of subsets of different ambient spaces.
\end{enumerate}
\end{lem}

\noindent
A family $(X_b)_{ b \in Y}$ of subsets of $K^n$ is { \bf definable (over $c$)} if both $Y$ and the set  $$\big\{(a,b) \in K^n \times Y : (a,b) \in X_b  \big\} $$ are definable (over $c$).
We note that if $(X_b)_{ b \in Y}$ is definable over $c$, then for each $b \in Y$, $X_b$ is definable over $(b,c)$ but not necessarily over $c$. 

For two families of subsets of $K^n$ which are definable (over $c$), we can choose  a combination, a fiberwise intersection, a fiberwise union and a fiberwise product of these two families to be definable (over $c$); the statement about fiberwise product can be generalized in an obvious way for two families of subsets of different ambient spaces.
A {\bf presentation} of  $X \subseteq K^n$ is a definable family $( H_\alpha)_{\alpha \in D}$ such that $$X\ =\ \bigcup_{\alpha \in D} H_\alpha\ \text{ and }\ D\subseteq F^k\ \text{ for some } k.$$ 
An {\bf algebraic presentation} $( V_\alpha)_{\alpha \in D}$ of $S\subseteq K^n$ is a presentation of $S$ such that for each $\alpha \in D$, $V_\alpha$ is $K$-algebraic.
If $ S \subseteq K^n$ has an algebraic presentation (which is definable over $c$), we say $S$ is { \bf algebraically presentable (over $c$)}; if $S\subseteq K^n$ has an algebraic presentation which is $0$-definable, we say $S$ is {\bf $0$-algebraically presentable}. If $S$ plays no important role, we sometimes use the term {\it algebraic presentation} without mentioning $S$. For the rest of this section, $S$ is an algebraically presentable subset of its ambient space. It is easy to observe that:

\begin{lem}
Suppose  $(V_\alpha)_{ \alpha \in D}$  and  $(V'_{\alpha'})_{ \alpha' \in D'}$ are algebraic presentations definable over $c \in K^m$.
We can choose a combination (fiberwise intersection, fiberwise union, fiberwise product) of $(V_\alpha)_{ \alpha \in D}$ and $(V'_{\alpha'})_{ \alpha' \in D'}$  to also be an algebraic presentation definable over $c$. 
\end{lem}

\noindent
An algebraically presentable set can be considered geometrically simple, and next we show that $0$-algebraically presentable sets are also syntactically simple.

\begin{lem} \label{APsimple1}
Suppose $ S \subseteq K^n$ is algebraically presentable over $c$. Then we can find an algebraic presentation $(V_\alpha)_{ \alpha \in D}$ and a system of polynomials $P$ in $\qq(c)[w,x]$ such that $V_\alpha = Z\Big( P\big(\chi(\alpha),x\big)\Big)$ for all $\alpha \in D$.
\end{lem}

\begin{proof}
Suppose $S$ has an  algebraic presentation $( W_\beta)_{ \beta \in E}$ definable over $c$. For each choice $\mathscr{C}$ of $k$ and a system $P$ of polynomials in $\qq(c)[w,x]$, define $R_{\mathscr{C}} \subseteq F^k \times E$ by  
$$(\alpha, \beta) \in R_{\mathscr{C}}\ \text{ if and only if }\ W_\beta \ =\  Z\Big( P\big(\chi(\alpha),x\big)\Big).$$
Then the relation $R_{\mathscr{C}}$ is definable and so are its projections $R^1_{\mathscr{C}}$ on $F^k$  and $R^2_{\mathscr{C}}$ on $E$. For each $\beta \in E $, any automorphism of $K$ fixing $\chi(F)$ and $c$ will also fix $W_\beta$. 
Therefore, for each $\beta \in E$, $W_\beta$ is definable  in the field sense over $\qq\big(c, \chi( \alpha)\big)$ for some $\alpha \in F$. Hence, there is a choice $\mathscr{C}$ as above such that $\beta \in R^2_{\mathscr{C}}$. There are countably many such choices $\mathscr{C}$.
By replacing $(F,K;\chi)$ by an elementary extension, if necessary, we can without loss of generality assume that the structure $(F,K;\chi)$ is $\aleph_0$-saturated.
Hence, there are choices $\mathscr{C}_1, \ldots, \mathscr{C}_l$ such that  $E$ is covered by $R^2_{\mathscr{C}_i}$ as $i$ ranges over $\{1, \ldots, l\}$. 

For $i \in \{1, \ldots, l\}$, obtain $k_i$ and $P_i$ from  the choice $\mathscr{C}_i$ and let $D_i= R^1_{\mathscr{C}_i} \subseteq F^{k_i}$. 
Set $D = D_1 \times \cdots \times D_l$, $P= P_1\cdots P_l$ and $V_\alpha = Z\Big( P\big(\chi(\alpha),x\big)\Big)$. It is easy to check that the family $(V_\alpha)_{ \alpha \in D}$ satisfies the desired requirements.
\end{proof}

\noindent
We have a slightly different version of the above lemma which will be used later.

\begin{lem}  \label{APsimple2}
Suppose $ S \subseteq K^n$ has an algebraic presentation $( W_\beta)_{ \beta \in E}$ definable over $c$. We can find an algebraic presentation $(V_\alpha)_{ \alpha \in D}$ and systems $P_1, \ldots, P_l$ of polynomials in $\qq(c)[w,x]$, such that  $(V_\alpha)_{ \alpha \in D}$ is equivalent to $( W_\beta)_{ \beta \in E}$, $D\subseteq F^k$ is the disjoint union of $D_1, \ldots, D_l$, each definable over $c$, and $ V_\alpha = Z\big( P_i(\chi(\alpha),x)\big)$ for $i\in \{ 1,\ldots, l\}$ and $\alpha \in D_i$. 
\end{lem}

\begin{proof}
We get the choices $\mathscr{C}_1, \ldots, \mathscr{C}_l$ in exactly the same way as in the first paragraph of the proof of the preceding lemma.
By adding extra variables, if needed, we can arrange that $k_1 = \cdots = k_l = k$ where $k_i$ is taken from the choice $\mathscr{C}_i$. We define $D_i$ inductively. For each $i \in { 1,\ldots, l} $, set
$$D_i \ =\  \bigg\{ \alpha \in F^k \backslash ( \bigcup_{j<i} D_j): \text{ there is } \beta \in E \text{ with } W_\beta = Z\Big( P_i\big(\chi(\alpha),x\big)\Big) \bigg \}.$$
Let $D = \bigcup_{i=1}^l D_i$, and $(V_\alpha)_{ \alpha \in D} $ be given by $V_\alpha = Z \big( P_i(\chi(\alpha),x)\big)$. It is easy to check that $(V_\alpha)_{ \alpha \in D}$ is the desired algebraic presentation.
\end{proof}

\noindent
Next, we prove that $F$ is $0$-stably embedded into $(F, K; \chi)$.
\begin{lem}\label{StablyEmbedednessForParameterFreeSet}
If $ D \subseteq F^k$ is $0$-definable, then it is $0$-definable in the field $F$.
\end{lem}

\begin{proof}
By changing the model if needed, we can arrange that $ (F, K; \chi) $ realizes all the $0$-types. By Stone's representation theorem, it suffices to show that if $ \alpha $ and $\alpha'$ are arbitrary elements in $ F^k$ with the same $0$-type in the field $F$, then they have the same $0$-type. Fix such $\alpha$ and $\alpha'$. As  $F$ is a model of $ \text{ACF} $, there is an automorphism of $F$ sending $ \alpha $ to $ \alpha'$. This automorphism can be extended to an automorphism of $(F, K;\chi) $ by Corollary~\ref{AutExt}, so $ \alpha $ and $ \alpha'$ have the same  $0$-type. 
\end{proof}

\begin{prop}
If $S\subseteq K^n$ is $0$-algebraically presentable, then we can find a formula $\varphi(s)$ in the language of rings and a system of polynomials $P \in \qq[w,x]$ such that $S$ is defined by 
$$ \exists s \Big(\varphi(s)\wedge P\big(\chi(s),x\big)=0 \Big).$$
\end{prop}

\begin{proof}
This follows from Lemma~\ref{APsimple1} and Lemma~\ref{StablyEmbedednessForParameterFreeSet}.
\end{proof}

\noindent
We next show that $0$-definable sets are just boolean combinations of $0$-algebraically presentable sets. Towards this, we need a number of lemmas.

\begin{lem} \label{SaturatedExtension}
The model $(F, K; \chi)$ has an elementary extension $(F', K'; \chi')$ such that $F'= F$ and $K'$ is $|F'|^+$-saturated as a model of $\ACF$.
\end{lem}

\begin{proof}
This follows from Corollary~\ref{ElementarilyEmbeddable}.
\end{proof}

\noindent
The following lemma is well known about $\ACF$. The proof is a consequence, for example, of the results in \cite{LouSmith}. 

\begin{lem}\label{DOE1}
Let $( X_b)_{b\in Y}$ be a family of subsets of $K^n$ definable ($0$-definable) in the field $K$. We have the following:

\begin{enumerate}
\item (Definability of dimension in families)

\noindent
the set $Y_k= \big\{b \in Y : \mr_K(X_b) =k\big\}$ is definable ($0$-definable) in the field $K$;
\item (Definability of multiplicity in families)

\noindent
the set $Y_{k, l}= \big\{b \in Y : \mr_K(X_b) =k,\  \md_K(X_b) = l \big\}$ is definable ($0$-definable) in the field $K$;
\item (Definability of irreducibility algebraic families)

\noindent
if $X_b$ is $K$-algebraic for all $b \in Y$, then $Y_{\mathrm{ired}}= \{b \in Y : X_b \text{ is irreducible} \}$ is definable ($0$-definable) in the field $K$.
\end{enumerate}
\end{lem}


\begin{cor}\label{DOE2}
Let $( X_b)_{b\in Y}$ be a definable (0-definable) family of subsets of $K^n$ with $X_b$ definable in the field $K$ for all $b \in Y$. Then we have the following:
\begin{enumerate}
\item (Definability of dimension in families)

\noindent
the set $Y_k= \big\{b \in Y : \mr_K(X_b) =k \big\}$ is definable ($0$-definable);
\item (Definability of multiplicity in families)

\noindent
the set $Y_{k, l}= \big\{b \in Y : \mr_K(X_b) =k,\  \md_K(X_b) = l \big\}$ is definable ($0$-definable);
\item (Definability of irreducibility in algebraic families)

\noindent
if $X_b$ is $K$-algebraic for all $b \in Y$, then $Y_{\mathrm{ired}}= \{b \in Y : X_b \text{ is irreducible} \}$ is definable ($0$-definable).
\end{enumerate}
\end{cor}

\begin{proof}
We first prove (1) for the definable case. Let  $( X_b)_{b\in Y}$ be a definable family as stated. For each $b \in Y$, there is a parameter free formula $\varphi(w, x)$ in the language of rings such that there is $c \in K^{k}$ with $X_b$ defined by $\varphi(c, x)$.
We note that there are only countably many parameter free formulas $\varphi(w,x)$ in the language of rings. 
By a standard compactness argument and a simple reduction we arrange that there is  a formula $\varphi(w,x)$ such that for any $b$ in $Y$, there is $c \in K^{k}$ such that $X_b$ coincides with $X'_c$  where $X'_d \subseteq K^n$ is defined by $\varphi(d, x)$ for $d \in K^k$.  With $Y_k$ as in the statement of the lemma, we have
$$Y_k\ =\ \big\{ b \in Y : \text{ there is } c\in K^{k} \text{ such that } X_b = X'_c \text{ and } \mr_K(X'_c) = k  \big\}. $$
The definability of $Y_k$ then follows from (1) of the preceding lemma.

For the 0-definable case, we can arrange that $(F, K; \chi)$ is $\aleph_0$-saturated and check that any automorphism of the structure fixing $( X_b)_{b\in Y}$ also fixes $Y_{k}$ for all $k$.
The statements (2) and  (3) can be proven similarly.
\end{proof}

\noindent
Towards obtaining the main theorem, we need the following auxiliary lemma.

\begin{lem}\label{QuanRedLem}
For $a \in K^n$, choose $V \subseteq K^n$ containing $a$ and definable in the field $K$ over $ \qq\big( \chi(F^\times)\big) $  such that $\big(\mr_K(V), \md_K(V)\big)$ is lexicographically minimized with respect to these conditions. Likewise, choose $V'\subseteq K^n$ for $a' \in K^n$.
If there are $ \alpha, \alpha' \in F^k$ of the same $0$-type in $F$ and a system $P$ of polynomials in $\qq[w,x]$  with $ V = Z\big(P( \chi( \alpha), x) \big)  $ and $V' =Z\big(P( \chi( \alpha'), x) \big)$, then  $ a$ and $a'$ have the same $0$-type.
\end{lem}

\begin{proof}
Using Lemma~\ref{SaturatedExtension}, we can arrange that $ K$ is $ |F|^{+}$-saturated as a model of $\ACF$. Suppose $a,a', V, V', \alpha, \alpha'$ and $P$ are as stated. Then we get an automorphism $\sigma_F$ of $F$ mapping $\alpha$ to $\alpha'$.
By Corollary~\ref{AutExt}, this can be extended to an automorphism  $ (\sigma_F, \sigma_K)$ of $ (F, K; \chi)$. In particular, 
$$\sigma_K: \chi(\alpha) \mapsto \chi(\alpha')\ \text{ and }\ \sigma_K(V) = V'.$$
Then $V'$ contains $\sigma(a)$, is defined over $ \qq\big( \chi(F^\times)\big) $ and $\big(\mr_K(V'), \md_K(V')\big)$ achieves the minimum value under these conditions.
Hence, for an algebraic  set $W \subseteq K^n$ definable in the field $K$ over $ \qq\big( \chi(F^\times)\big) $, 
$$\sigma(a) \in W\ \text{ if and only if }\ \big(\mr_K(V'\cap W), \md_K(V' \cap W)\big)\ =\ \big(\mr_K(V'), \md_K(V')\big). $$
By the choice of $V'$, exactly the same statement holds when $\sigma(a)$ replaced with $a'$. 
By the quantifier elimination of $\ACF$, $ \sigma(a) $ and $ a'$ have the same type over $ \qq\big( \chi(F^\times)\big) $ in the field $K$.
As $ K$ is $ |F|^{+}$-saturated, there is an automorphism $\tau_K$ of $K$ fixing $\qq\big(\chi(F^\times)\big)$ pointwise and mapping $\sigma(a)$ to $a'$.
It is easy to check that $(\sigma_F, \tau_K \circ \sigma_K)$ is an automorphism of $(F, K; \chi)$ mapping $a$ to $a'$. Therefore, $a$ and $a'$ have the same $0$-type.
\end{proof}

\begin{thm} \label{QuanRed}
If $X \subseteq K^n$ is $0$-definable, then $X$ is a boolean combination of $0$-algebraically presentable subsets of $K^n$.  
\end{thm}

\begin{proof}
We say $a, a' \in K^n$ have the same 0-ap-type if they belong to the same $0$-algebraically presentable sets. 
By changing the model, if needed, we can arrange that $ (F, K; \chi) $ realizes all the $0$-types. 
By Stone's representation theorem, it suffices to show that if $ a$ and $ a' $ are arbitrary elements in $  K^n $ with the same $0$-ap-type then they have the same $0$-type. 

\noindent
Fix $a$ and $a'$ in $K^n$ with the same  $0$-ap-type. 
Choose $V \subseteq K^n$ containing $a$ and definable in the field $K$ over $ \qq\big( \chi(F^\times)\big) $  such that $$\big(\mr_K(V), \md_K(V)\big) \text{ is lexicographically minimized with respect to these conditions}. $$ 
Moreover, pick $k$, $D \subseteq F^k$ $0$-definable in the field $F$, $\alpha \in D$ and a system $P$ of polynomials in $\qq[w,x]$ such that $V = Z\big( P( \chi( \alpha), x)\big)$ and $$\big(\mr_F(D), \md_F(D) \big) \text{ is minimized under these conditions}.$$
We will find $\alpha'$ and $ V'$ in order to use Lemma~\ref{QuanRedLem}. 
Set $$E\  =\ \bigg\{ \beta \in D : \text{ if } V_\beta = Z\Big( P\big( \chi( \beta), x\big)\Big) \text{, then } \big(\mr_K(V_\beta), \md_K(V_\beta)\big)=\big(\mr_K(V), \md_K(V)\big)\bigg\} .$$ 
We note that $E $ is $0$-definable by Corollary~\ref{DOE2}, and so by Lemma~\ref{StablyEmbedednessForParameterFreeSet}, $ E$ is also $0$-definable in the field $F$. 
As $\alpha$ is in $E$, 
$$ \big(\mr_F(E), \md_F(E) \big) \ = \ \big(\mr_F(D), \md_F(D) \big) $$ by the choice of $D$. 
With $S = \big( Z\big(P(\chi(\beta),x)\big)\big)_{\beta \in E}$, we have $a\in S$, and so we also have $a' \in S$ since $a$ and $ a'$ have the same 0-ap-type. Hence, there is $\alpha'\in E$ such that $a'$ is an element of $V' = Z\big(P(\chi(\alpha'),x)\big)$.

We next verify that $ \alpha'$ and $ V'$ satisfy the conditions of Lemma~\ref{QuanRedLem}. It will then follow that $a$ and $a'$ have the same $0$-type. We first check that
$$\big(\mr_K(V'), \md_K(V')\big)\ =\  \min\Big\{ \big(\mr_K(W'), \md_K(W')\big) : W' \subseteq K^n \text{ is } K\text{-algebraic, } a' \in W' \Big\}.$$
As $ \alpha'$ is in $E$, $$\big(\mr_K(V'), \md_K(V')\big)\ =\ \big(\mr_K(V), \md_K(V)\big).$$ 
Suppose towards a contradiction that there is an irreducible algebraic set $ W' \subseteq K^n$ containing $a'$ with $$ \big(\mr_K(W'), \md_K(W')\big) \ <_{\text{lex}}\ \big(\mr_K(V'), \md_K(V')\big).$$ 
We can do the same construction as above in the reverse direction to get $ W''$ with 
$$ \big(\mr_K(W''), \md_K(W'')\big)\ <_{\text{lex}} \  \big(\mr_K(V), \md_K(V)\big)$$ containing $a$, a contradiction to the choice of $V$. 
We next check that $\alpha$ and $\alpha'$ have the same $0$-type in the field $F$. Suppose otherwise. Let $ D'$ be the  smallest 0-definable $F$-algebraic set containing $\alpha'$.
Then $$ \big(\mr_F(D'), \md_F(D') \big) \ <_{\text{lex}}\ \big(\mr_F(D), \md_F(D) \big).$$
Do the same construction in the reverse direction again to get $ \alpha'' \in D'$ such that $a $ satisfies $ P\big( \chi( \alpha''), x\big) =0$.
If $ D''$ is the smallest 0-definable $F$-algebraic set containing $\alpha''$, then $$ \big(\mr_F(D''), \md_F(D'') \big)\ \leq_{\text{lex}}\  \big(\mr_F(D'), \md_F(D') \big)\ <_{\text{lex}} \ \big(\mr_F(D), \md_F(D) \big), $$ a contradiction to our choice of $D, \alpha$ and $P$.
\end{proof}

\noindent
Suppose $D \subseteq F^k$ is definable. By Proposition~\ref{transferRemark}, $D$ can be identified with $\chi(D) $, which has a simple description by the preceding theorem. In the rest of the section, we give an improvement of the above result for this special case. 
For a system  $P$ in $K[w]$, we abuse the notation and let $Z\big(P(\chi(s))\big) \subseteq F^k$ be the set defined by  $P\big(\chi(s)\big)=0$.

\begin{lem}
For each $k$ there is a system $Q$ of polynomials in $ F[s] $ such that the set defined by $ \chi(s_1) + \cdots+ \chi(s_k) =0 $ is $Z(Q)$.
\end{lem}
\begin{proof}

For $ I \subseteq \{1, \ldots, k\} $, let $ \sum_{i\in I} \chi(s_i) \ndeq 0 $ denote the system which consists of $\sum_{i\in I} \chi(s_i) = 0 \text{ and }\sum_{i\in I'} \chi(s_i) \neq 0$ for each non-empty proper subset $I'$ of $I$. By Mann's theorem, there are $\alpha^{(1)}, \ldots, \alpha^{(l)}$ in $F^I$, such that the set defined by $ \sum_{i\in I} \chi(s_i) \ndeq 0  $ precisely  consists of  $ \beta \alpha^{(j)}$  with $\beta \in F^\times$ and $j \in \{1, \ldots, l \}$. Hence, if $ I \subseteq \{1, \ldots, k\} $, then there is a system $ Q_{I} $ of polynomials in $F[s]$  such that the set defined by $ \sum_{i\in I} \chi(s_i) \ndeq 0 $ is $ Z(Q_I)$.

Consider all the partitions $ \mathscr{P} $ of the set $\{1,\ldots, k\}$ into non-empty subsets. Then we have the set defined by $\chi(s_1) + \cdots + \chi(s_k) = 0 $ is $ \bigcup_{\mathscr{P}} \bigcap_{I \in \mathscr{P}} Z(Q_I)$. Note that finite unions and finite intersections of $F$-algebraic sets are again $F$-algebraic. Thus, we can find a system $Q$ of polynomials in $ F[s] $ as desired.
\end{proof}

\begin{lem} \label{Continuity}
The map $ \chi: F^k \to K^k $ is continuous. 
\end{lem}

\begin{proof}
For the statement of the lemma, we need to show that if $ V\subseteq K^k$ is $K$-closed then $ \chi^{-1}(V) $ is $F$-closed. It suffices to show that if $P$ is in $ \qq[w, x] $ and $ a$ is a tuple of elements in $K$, then  $ Z\big(P(\chi(s), a)\big) $ is $F$-algebraic. 
Choose a linear basis $ B$ of $\qq\big(\chi(F^\times), a\big)$ over $\qq\big(\chi(F^\times)\big)$. 
Then 
$$P\big(\chi(s), a\big) \ = \ P_1\big(\chi(s)\big)b_1+\cdots+ P_m\big(\chi(s)\big)b_m $$ where $P_i $ has coefficients in $ \qq(\chi(F^\times))$, $b_i \in B$ for $i \in \{1, \ldots, m\}$, and $b_i \neq b_j$ for distinct $i,j \in \{ 1 , \ldots, m\}$. 
Therefore,  $P\big(\chi(s), a\big) = 0$ is equivalent to $ P_i\big(\chi(s)\big)=0$ for all $i \in \{1, \ldots, m \}$. 
Furthermore, for each $i \in \{ 1, \ldots,m \} $, $ P_i\big(\chi(s)\big) =0 $ is equivalent to an equation of the form 
$$ \chi\big(M_1(s, \alpha)\big)+\cdots+\chi\big(M_{l_i}(s, \alpha)\big)\ =\ 0 $$ 
where $ \alpha$ is a tuple of elements in $F$, and $ M_j$ is a monomial for $j \in \{1, \ldots, l_i\}$. By the result of the preceding lemma, for each $i \in \{1, \ldots, m\}$, the polynomial equation $P_i\big(\chi(s)\big)=0$ is equivalent to a system $ Q_i\big( M_1(s, \alpha), \ldots, M_{l_i}(s, \alpha)\big) =0 $. Thus, $Z\big(P(\chi(s), a)\big) = \bigcap_{i=1}^k Z\big(P_i(\chi(s))\big)$ is $F$-algebraic.
\end{proof}

\begin{thm} \label{StablyEmbbed}
Let $D$ be a subset of $F^k$. If $D$ is definable, then $D$ is definable in the field $F$. Moreover, when $D$ is $0$-definable, $D$ is $0$-definable in the field $F$. If $D = \chi^{-1}(V)$ with $K$-algebraic $V \subseteq K^n$, then $D$ is an $F$-algebraic set. Moreover, when $V = Z(P)$ with $P$ a system in $\zz[w]$, $D = Z(Q)$ with $Q$ a system in $\zz[s]$. 
\end{thm}

\begin{proof}

We prove the first assertion. It suffices to show that if $ X \subseteq K^k$ is definable, then $ \chi^{-1}(X) $ is definable in the field $F$.
By Theorem~\ref{QuanRed}, we only need to show that if $S \subseteq K^{k+m}$ is $0$-algebraically presentable and $X = \big\{a : (a,b) \in S \big\}$ with $b \in K^m$ then  $ \chi^{-1}(X) $ is definable in the field $F$. 
It is easy to see that $ X $ is defined by a formula of the form 
$$ \exists t \Big( \varphi(t) \wedge P\big(w, \chi(t)\big) = 0 \Big)\ \text{ where } P \text{ is a system of polynomials in } K[w, x].$$
Let $V$ be  $ Z(P)$. Then by the preceding lemma, $ \chi^{-1}(V) $ is $Z(Q)$ for some system $Q$ in $F[s, t]$. Hence, $ \chi^{-1}(X)$, which is defined by  $ \exists t \big( \varphi(t) \wedge P(\chi(s),\chi(t)) =0 \big) $, is also defined by $ \exists t \big( \varphi(t) \wedge Q(s,t) =0 \big)  $. Thus, $ \chi^{-1}(X)$ is definable in the field $F$ as desired.  The second assertion is just Lemma~\ref{StablyEmbedednessForParameterFreeSet}. The third assertion is Lemma~\ref{Continuity}. The forth assertion follows from the second and third assertions. 
\end{proof}

%% file: 6.GenericRankAndDegree.tex
\section{Definable sets II}


\noindent
We keep the notation conventions of the preceding section. Furthermore, We assume that, with possible decorations, $V, W,  C$ are $K$-algebraic subsets of their ambient spaces and $C$ is $K$-irreducible; also with possible decorations, $S$ is an algebraically presentable subset of its ambient space.

The goal of this section is to obtain an ultimate description of definable sets in $(F, K; \chi)$ which allows us to define good notions of dimension and multiplicity. 
This is done in several steps by introducing intermediate descriptions which gradually increase our geometric understanding of definable sets.
This analysis is complicated by the fact that not all definable sets are algebraically presentable; indeed, algebraically presentable sets are definable with only existential formulas while $\TN_p$ is not model complete. If we try to replace ``algebraic'' with ``constructible'', we will still run into the same problem. Therefore, we will need to take one step further.

We call $T \subseteq K^n$ a {\bf pseudo-constructible set} (or {\bf pc-set}) if there are $ V, S \subseteq K^n$   such that $T= V \backslash S$. A pc-set is clearly definable. If $V'$ is the closure of $T$ in the $K$-topology then $T= V' \backslash (S \cap V')$, and $S \cap V'$ is also algebraically presentable. Hence, if $T$ is a pc-set, there is a choice of $V, S$ such that $V$ is the closure of $T$ in the $K$-topology. Throughout the rest of this section, $T$ with possible decorations is a pc-subset of its ambient space. If $T= V \backslash S$, and $S$ has an algebraic presentation with only finitely many elements, then $T$ is a constructible set in the $K$-topology. This section is based on the observation that we can almost pretend pc-sets are constructible sets in the $K$-topology. The underlying reason is the following:

\begin{lem}\label{Smallness1}
Suppose $S \subseteq K^n$ has algebraic presentation $\{W_\beta\}_{\beta \in E}$ and $C$  is a subset of $S$. Then $C$ is a subset of $W_\beta$ for some $\beta \in E$.
\end{lem}

\begin{proof}
Suppose $S, C$ and $\{W_\beta\}_{ \beta \in E}$ are as stated. By Corollary~\ref{SaturatedExtension}, we can arrange that $K$ is $|F|^+$-saturated as a field. As $|E| \leq |F|$, $C \subseteq \bigcup_{\beta \in E} W_\beta $ implies $C$ is a subset of a union of finitely many elements in $\{W_\beta\}_{ \beta \in E}$. Since $C$ is irreducible, $C \subseteq W_\beta$ for some $\beta \in E$.
\end{proof}

\begin{cor} \label{Smallness2}
If $C, S, S' \subseteq K^n$ are such that $C \subseteq S\cup S'$, then either $C \subseteq S$ or $C \subseteq S'$.
\end{cor}

\noindent
The above lemma allows us to analyze pc-sets through their closures.

\begin{cor} \label{Smallness3}
Suppose $C, S$ are subsets of $K^n$. Then  $C \backslash S$ has closure $C$ if and only if $C \nsubseteq S$.
\end{cor}

\begin{proof}
 Suppose $C, S \in K^n$ are as stated. The forward direction is clear. Suppose $ \{W_\beta \}_{\beta \in E}$ is an algebraic presentation of $S$ and $C \nsubseteq S$. Let $V$ be the closure of $C \backslash S$. Then $C \backslash S \subseteq V \subseteq C \subseteq V \cup S $. By the preceding corollary, the last inclusion implies either $C \subseteq V$ or $C \subseteq S$. By assumption $C \nsubseteq S$, so $C \subseteq V$. Thus, $C = V$ as desired.
\end{proof}

\begin{cor} \label{Smallness4}
Suppose $V=C_1\cup\ldots\cup C_k, X = V_1 \backslash S_1 \cup \ldots V_l \backslash S_l$ are subsets of $K^n$, and $X$ is a subset of $ V$. Then $V$ is the  $K$-closure of $X$ if and only if for each $i \in \{1, \ldots,k\}$, there is $j \in  \{1, \ldots, l\} $ such that $C_i \subseteq V_j$ and $C_i \nsubseteq S_j$.
\end{cor}

\begin{proof}
We have $V$ is the $K$-closure of $X$ if and only if for each $i \in {1, \ldots, k}$, $C_i$ is the closure of $C_i \cap X$. For each $i$, as $C_i$ is $K$-irreducible, $C_i$ is the closure of $C_i \cap X$ if and only if there is  $j \in \{1, \ldots, l\} $ such that $C_i$ is the closure of $ C_i \cap ( V_j \backslash S_j)= (C_i \cap V_j) \backslash S_j$. By the preceding corollary, this happens if and only if $C_i \cap V_j =C_i$ and $C_i \nsubseteq S_j$.
\end{proof}

\begin{cor} \label{ClosurePreserve}
Suppose $T, V$ are subsets of $K^n$. If $V$ is the closure of $T$ in the $K$-topology, then this also holds in any elementary extension of the model.
\end{cor}

\begin{proof}
This follows immediately from the preceding corollary.
\end{proof}

\noindent
The collection of pc-sets is not closed under complement. The following definitions allow us to overcome this limitation.
Suppose $T,T'$ are subset of $K^n$ and $V'$ is the closure of $T'$ in the $K$-topology. We define $T \capdot T'$, to be $T \cap V'$. Note that this definition is not symmetric. Set $T \dotminus T'$ to be $ T \backslash V'$. Clearly, $T = (T \capdot T') \cup (T \dotminus T')$.
By a routine manipulation of formulas we get:
 
\begin{lem} 
If $T, T' \subseteq K^n$ then $T \cap T'$, $T \capdot T'$, $T \dotminus T'$ are pc-sets. Likewise, for $T \subseteq K^n$, $T'\subseteq K^{n'}$, we have that $ T \times T' \subseteq K^{n+n'}$ is also a pc-set.
\end{lem}

\noindent
Suppose $T$, $T'$, $V$, $V'$ are subsets of $K^n$, and $V$ is the $K$-closure of $T'$, $V'$ is the $K$-closure of $T'$. The {\bf$K$-Morley rank} of $T$ is defined by $\mr_K(T)= \mr_K(V)$; the {\bf$K$-Morley degree}  of $T$ is defined by $\md_K(T)=\md_K(V)$.

We say $T$ is {\bf almost a subset} of $T'$, if  $\mr_K(T) = \mr_K(T') =\mr_K(T \cap T')$ and $\md_K(T\cap T') =\md_K(T)$, and denote it by $T  \subsim T'$. We note that in our definition $T\subseteq T'$ does not imply $T \subsim T'$ as we might have $\mr_K(T) < \mr_K(T')$. The definition is given in this way to simplify the notation in the case of $\mr_K(T) =\mr_K(T')$, which is our focus.
We say $T$ and $T'$ are {\bf almost equal}, denoted by $T\sim T'$, if $T \subsim T'$ and $T' \subsim T$. We say $T$ and $T'$ are { \bf almost disjoint}, denoted by $T \simperp T'$ if $\mr_K(T)=\mr_K(T')$ and $\mr_K(T\cap T')< \mr_K(T)$.

The following facts are very natural analogues of what we expect to be true about constructible sets in $K$-topology. All are either straightforward from the definitions or easy consequences of Corollaries~\ref{Smallness2} and~\ref{Smallness4}.
\begin{prop} \label{lies}
Suppose $T$, $T_1$, $T_2$, $V_1$, $V_2$ are subsets of $ K^n$ and $T'$, $T_1'$, $T_2'$ are subsets of $K^{n'}$, and $V_1$,$V_2$ are respectively the closure of $T_1$,$T_2$ in the $K$-Zariski topology. Then we have the following:
\begin{enumerate}
\item if $T_1$ is a subset of $T_2$, then either $\mr_K(T_1) \leq \mr_K(T_2)$ or $\mr_K(T_1) = \mr_K(T_2)$ and $\md_K(T_1) \leq \md_K(T_2)$;
\item if $T = T_1 \cup T_2$, then $\mr_K(T) = \max\{ \mr_K(T_1), \mr_K(T_2)\} $;
\item if $T = T_1 \cup T_2$, $T_1 \simperp T_2$, then $\md_K(T) = \md_K(T_1)+\md_K(T_2)$;
\item the relation $\subsim$ is transitive; the relation $\sim$ is an equivalent relation;
\item if $\mr_K(T_1)=\mr_K(T_2)= \mr_K(T _1 \cap T_2)$, then $T_1 \cap T_2 \sim T_1 \capdot T_2 \sim V_1 \cap V_2$;
\item if $\mr_K(T_1) = \mr_K(T_2)$, then exactly one of the following can happen: $T_1 \simperp T_2$, $T_1 \subsim T_2$ or $\mr_K(T_1)=\mr_K(T_2)= \mr_K(T _1 \capdot T_2) = \mr_K( T_1 \dotminus T_2)$;
\item if $T_1 \simperp T_2$, then $\mr_K(T_1 \capdot T_2) < \mr(T_1) $ and $T_1 \dotminus T_2 \sim T_1$;
\item if $T_1 \subsim T_2$, then $ T_1 \cap T_2 \sim T_1 \capdot T_2 \sim T_1$;
\item if $\mr_K(T_1)=\mr_K(T_2)= \mr_K(T _1 \capdot T_2) = \mr_K( T_1 \dotminus T_2)$, then 
$$\md_K(T _1 \capdot T_2)+\md_K(T_1 \dotminus T_2) \ =\ \md_K(T_1);$$
\item if $\mr_K(T_1)= \mr_K(T_2) = \mr_K(T_1 \cap T_2)$ and $T'_1 \sim T_1, T'_2 \sim T_2$, then $T'_1 \cap T'_2 \sim T_1 \cap T_2$. The same conclusion holds if as we replace all appearances of $\cap$ in the previous statement with one of $\cup, \capdot, \dotminus$;
\item $\mr_K( T \times T') = \mr_K(T) + \mr_K(T')$, $ \md_K( T \times T') = \md_K(T)\md_K(T')$;
\item if $T_1  \sim T_2$ and $T_1' \sim T'_2$, then $T_1 \times T'_1 \sim T_2 \times T'_2$; if $T_1 \sim T_2$ and $T'_1 \simperp T'_2$, then $T_1 \times T'_1 \simperp T_2 \times T'_2$; if   $T_1 \simperp T_2$ and $T'_1 \simperp T'_2$, then $T_1 \times T'_1 \simperp T_2 \times T'_2$.
\end{enumerate}
\end{prop}

\noindent
We now introduce our first intermediate description of definable sets in a model of $\TN_p$.
A {\bf pseudo constructible presentation}  (or {\bf pc-presentation}) of a set $X \subseteq K^n$ is a definable family $ \{ T_\alpha\}_{ \alpha \in D} $ where $D \subseteq F^k$ for some $k$, such that $X = \bigcup_{ \alpha \in D} T_\alpha $.  We will also talk of a pc-presentation without mentioning $X$; by that we mean a pc-presentation for some $X$, but $X$ plays no important role. If $X$ has a pc-presentation then $X$ is definable. We will show that the converse is also true. The following is immediate:

\begin{lem}
Suppose  $\{T_\alpha\}_{ \alpha \in D}$  and  $\{T'_{\alpha'}\}_{ \alpha' \in D'}$ are pc-presentations definable over $c \in K^m$. A
fiberwise intersection (respectively fiberwise union, disjoint combination or fiberwise product) of $\{T_\alpha\}_{ \alpha \in D}$ and $\{T'_{\alpha'}\}_{ \alpha' \in D'}$ can be chosen to be also a pc-presentation definable over $c$. 
\end{lem}

\begin{prop} \label{pc-pres}
Every definable subset of $K^n$ has a pc-presentation.
\end{prop}

\begin{proof}
We will first show the statement for $X \subseteq K^n$ of the form $ S \backslash S'$. By definition $ S $ has an algebraic presentation $ \{ V_\alpha\}_{ \alpha \in D} $. For $ \alpha \in D$, let $ T_\alpha = V_\alpha \backslash S' $. It can be easily checked that $ \{ T_\alpha\}_{ \alpha \in D} $ is a pc-presentation of $ S \backslash S'$. By Theorem~\ref{QuanRed}, every definable subsets of $K^n$ can be written as a finite union of sets of the form $ S \backslash S'$ where $S, S' \subseteq K^n$. By the preceding lemma and Lemma~\ref{fiber}, the collection of sets having a pc-presentation is closed under finite union. 
The conclusion follows.
\end{proof}

\noindent
The next proposition allows us to define a geometrical invariant of a definable set based on a choice of its pc-presentation and yet independent of such choice.

\begin{prop} \label{Kgeorank}
If $X \subseteq X' \subseteq K^n$ are definable, $X$ has a pc-presentation $\{ T_\alpha\}_{ \alpha \in D}$ and $X'$ has a pc-presentation $\{ T'_{\alpha'}\}_{ \alpha' \in D'}$, then $\max_{\alpha \in D} \mr_K(T_\alpha) \leq \max_{\alpha' \in D'} \mr_K(T'_{\alpha'}) $. 
\end{prop}

\begin{proof}
Suppose $X$, $X'$, $\{ T_\alpha\}_{ \alpha \in D}$, $\{ T'_{\alpha'}\}_{ \alpha' \in D'}$ are as given. Let $\alpha$ be such that $\mr_K(T_{\alpha}) =\max_{\beta \in D}\mr_K(T_\beta) $. We can arrange that $T_\alpha = V_\alpha \backslash S_\alpha$ with $V_\alpha$ the $K$-closure of $T_\alpha$; let $C_\alpha$ be one of the the components of $V_\alpha$ with $K$-dimension $\mr_K(T_\alpha)$. Then by Lemma~\ref{Smallness4}, if $\{W_\beta\}_{\beta \in E}$ is an algebraic presentation of $S_\alpha$, 
$$\mr_K(C_\alpha \cap W_\beta)\ <\ \mr_K(T_\alpha)\ \text{ for each }\ \beta \in E.$$
Now suppose for all $\alpha' \in D'$, $\mr_K( T'_{\alpha'})< \mr_K(T_\alpha)$. For each $\alpha' \in D'$, let $V'_{\alpha'}$ be the $K$-closure of $T'_{\alpha'}$. We note that the family $\{ V'_{\alpha'}\}_{ \alpha' \in D'}$ has cardinality at most $|F|$; also, by Corollary~\ref{SaturatedExtension}, we can arrange that $K$ is $|F|^+$-saturated as a field. By dimension comparison, $C_\alpha$ is not a subset of a union of finitely many elements of $\{ V'_{\alpha'}\}_{ \alpha' \in D'}$ and finitely many elements of $\{W_\beta \cap C_\alpha\}_{\beta \in E}$. Thus $C_\alpha$ is not a subset of the union of $\{ V'_{\alpha'}\}_{ \alpha' \in D'}$ and $\{C_\alpha \cap W_\beta\}_{\beta \in E}$. This implies $T_\alpha$ is not a subset of the union $\{ V'_{\alpha'}\}_{ \alpha' \in D'}$, a contradiction; the conclusion follows.
\end{proof}

\begin{cor}
If $X \subseteq K^n$ has pc-presentations $\{ T_\alpha\}_{ \alpha \in D}$ and $\{ T'_{\alpha'}\}_{ \alpha' \in D'}$, then $\max_{\alpha \in D} \mr_K(T_\alpha) = \max_{\alpha' \in D'} \mr_K(T'_{\alpha'}) $.
\end{cor}

\noindent
Suppose $X\subseteq K^n$ has  a pc-presentation $\{ T_\alpha\}_{ \alpha \in D}$. Then the {\bf $K$-geometric rank} of $X$, denoted by $ \gr_K(X)$, is defined to be  $\max_{\alpha \in D} \mr_K(T_\alpha)$. The following is also immediate from the previous proposition:

\begin{cor}
Suppose $X, X' \subseteq K^n$ are definable. Then we have the following:
\begin{enumerate}
\item  if $X \subseteq X'$, then $\gr_K(X) \leq \gr_K(X')$;
\item $\gr_K(X \cup X') = \max\{ \gr_K(X), \gr_K(X') \}$.
\end{enumerate}
\end{cor}

\medskip \noindent
Next, we introduce the second intermediate description of definable sets in a model of $\TN_p$.
A family $\{T_b\}_{b \in Y}$ of pc-sets is {\bf essentially disjoint} if for any $b, b' \in Y$ such that $ \mr_K(T_b) =\mr_K(T_{b'}) = \max_{b \in Y} \mr_K(T_b)$, we have  either $T_b \sim T_{b'}$ or  $T_b \simperp T_{b'}$. An essentially disjoint pc-presentation is a pc-presentation which is essentially disjoint as a family of pc-sets.

\begin{lem}
Suppose  $\{T_b\}_{ b \in Y}$  and  $\{T'_{b'}\}_{ b' \in Y'}$ are families of pc-sets. Then an arbitrary fiberwise intersection (respectively fiberwise product) of $\{T_b\}_{ b \in Y}$  and  $\{T'_{b'}\}_{ b' \in Y'}$ is also essentially disjoint.
\end{lem}
\begin{proof}
This follows from (10), (11) and (12) of Lemma~\ref{lies}.
\end{proof}

\noindent
Towards replacing pc-presentation with essentially disjoint pc-presentation in Proposition~\ref{pc-pres} we need the following auxiliary result:

\begin{lem} \label{DefinabilityOfClosure}
Suppose $\{T_b\}_{b \in Y}$ is a definable family of pc-subsets of $K^n$ and for each $b \in Y$, $V_b$ is the closure of $T_b$ in the $K$-topology. Then the family  $\{V_b\}_{b \in Y}$ is definable.

\end{lem}

\begin{proof}

Suppose $\{T_b\}_{b \in Y}$ and $\{V_b\}_{b \in Y}$ are as in the assumption of the lemma. Let $\mathscr{C}$ be a choice of systems $Q^{(1)}, \ldots, Q^{(l)}$  in $\qq[w,x]$, a system $P'$ in $\qq[w',x]$, a system $P''$ in $\qq[w'',x, z'']$, a parameter free $L_r$-formula $\phi''(u'',v'')$ with  variables in the first sort  and $|v''|=|z''|$. We note that there are countably many such choices $\mathscr{C}$. We define a relation $R_\mathscr{C} \subseteq K^n \times Y \times F^{|v''|} \times K^{|w|}\times K^{|w'|} \times K^{|w''|}$ as follows. For $a \in K^n$, $b \in Y$, $\gamma'' \in F^{|v''|}$, $c \in K^{|w|}$, $c' \in K^{|w'|}$, and $c'' \in K^{|w''|}$, $R_\mathscr{C}(a, b, \gamma'', c, c', c'')$ holds if and only if the following conditions hold:
\begin{enumerate}[(a)]
\item $T_b = V'_{c'} \backslash S''_{ \gamma'', c''}$ where $V'_{c'}$ is the zero set of $P'(c',x)$ and $S''_{\gamma'', c''}$ is the set defined by $\exists v''\big( \phi''(\gamma'', v'') \wedge P''(c'', x, \chi(v'')) =0\big)$;
\item $Z\big( Q^{(1)}(c, x)\big), \ldots, Z\big( Q^{(l)}(c, x)\big) $ are irreducible;
\item $\bigcup_{i=1}^l Z\big( Q^{(i)}(c, x)\big) $ is the closure of $V'_{c'} \backslash S''_{\gamma'', c''}$ in $K$-topology, where $V'_{c'}, S''_{\gamma'', c''}$ are the same as in (a);
\item $a$ is in  $ \bigcup_{i=1}^l Z\Big( Q^{(i)}(c, x)\Big) $.
\end{enumerate}
We note that (a), (d) are clearly definable, (b) is definable as irreducibility is definable in families in models of $\text{ACF}$, and under the condition that (b) holds, (c) is definable by Corollary~\ref{Smallness4}. Hence, $R_\mathscr{C}$ is a definable relation. Let $R^2_\mathscr{C} \subseteq Y$ be the projection of $R_\mathscr{C}$ on $Y$, $R^{1,2}_\mathscr{C}$ be the projection of $R_\mathscr{C}$ on $K^n \times Y$. Then $R^2_\mathscr{C}, R^{1,2}_\mathscr{C} $ are also definable. We also note that  if  $b \in R^2_\mathscr{C}$ then $V_b$ is precisely the set $\big\{ a \in K^n : (a, b) \in R^{1,2}_\mathscr{C} \big\}$.

For each $b \in Y$, there is a choice $\mathscr{C}$ as above such that $b \in R^2_{\mathscr{C}}$. By a standard compactness argument, there are finitely many choices $\mathscr{C}_1, \ldots, \mathscr{C}_k$ as above such that for any $b \in Y$, there is $i \in \{ 1, \ldots, k\} $ such that $b \in R^2_{\mathscr{C}_i}$. Then the family $\{V_b\}_{b \in Y}$ as a subset of $K^n \times Y$ consists of $(a, b)$ such that for some $i \in \{1, \ldots k\}$, $(a, b) \in R^{1,2}_{\mathscr{C}_i}$.  Thus $\{V_b\}_{b \in Y}$ is definable.
\end{proof}

\begin{cor} \label{Definability}
Suppose $\{T_b\}_{b \in Y}, \{T'_{b'}\}_{b' \in Y'}$ are definable families of pc-subsets of $K^n$. The following sets are definable:
\begin{enumerate}
\item for $k \in \nn$, the set $\big\{ b \in Y : \mr_K(T_b) = k \big\}$;
\item for $k,l \in \nn$, the set $\big\{ b \in Y : \mr_K(T_b) = k,  \md_K(T_b) = l \big\}$;
\item the set $ \big\{ (b,b') \in Y \times Y' : \mr_K(T_b) \leq \mr_K( T'_{b'}) \big\}$;
\item the sets obtained by replacing $\leq$ in (3) with $<, =$ or replacing $\mr_K$ with $\md_K$;
\item the set $ \big\{ (b,b') \in Y \times Y' : T_b \subsim T'_{b'} \big\}$;
\item the sets obtained by replacing $ \subsim$ with $\sim$ and $\simperp$ in (5). 
\end{enumerate}
\end{cor}

\begin{proof}
All the above statements have similar proof ideas, so we will only provide the proof of the first statement as an example. Suppose $\{T_b\}_{b \in Y}$ is as in the statement and let $\{V_b\}_{b \in Y}$ be as in the preceding lemma. Then $\big\{ b \in Y \mid \mr_K(T_b) = k \big\}$ is by definition the same as $\big\{ b \in Y \mid \mr_K(V_b) = k \big\}$. The desired conclusion follows from the preceding lemma and the fact that the Morley rank is definable in family in a model of ACF (see Lemma~\ref{DOE1}).
\end{proof}

\noindent
Suppose $\{T_\alpha\}_{\alpha \in D}$ is a pc-presentation. The {\bf primary index set} $\hat{D}$ of $\{T_\alpha\}_{\alpha \in D}$ consists of   $\alpha \in D$ such that $\mr_K(T_\alpha) = \max_{ \alpha \in D} \mr_K(T_\alpha)$.

\begin{prop} \label{essdispre}
Every definable set has an essentially disjoint pc-presentation.
\end{prop}
 
\begin{proof}
Let $ \{ T_\alpha\}_{ \alpha \in D}$ be a pc-presentation of $X$, $r= \max_{ \alpha \in D} \mr_K(T_\alpha)$ and  $\hat{D}$ be the primary index set of $ \{ T_\alpha\}_{ \alpha \in D}$. Let  $ \hat{D}_1 $ be the set $$\{ \alpha \in \hat{D}: \text{ for all } \beta \in \hat{D}, \text{either } T_\alpha \sim T_\beta \text{ or } T_\alpha \simperp T_\beta\} $$ and let $ \hat{D}_2 = D\backslash \hat{D}_1$. Set $l=0$ if $\hat{D}_2$ is empty and $l = \max_{ \alpha \in \hat{D}_2} \md_K(T_\alpha)$ otherwise. We note that $\big\{ \md_K(T_\alpha):  \alpha \in \hat{D}_2 \big\}$ has an upper bound by the preceding lemma and a standard compactness argument.

We now make a number of reductions. If $r = 0$, the statement of the proposition is immediate. Towards using induction,  assume $r>0$ and we have proven the statement for all $X'$ with a pc-presentation $ \{ {T'}_{\alpha'}\}_{ \alpha' \in D'}$ such that for similarly defined $r'$, we have $r'<r$. Note that $\hat{D}, \hat{D}_1$ and $\hat{D}_2$ are definable by the preceding lemma. Let $X^-$ and $\hat{X}$ be the definable subsets of $K^n$ given by:
$$X^-\ =\ \bigcup_{ \alpha \in D\backslash \hat{D}}T_\alpha\ \text{ and }\ \hat{X}\ =\ \bigcup_{ \alpha \in \hat{D}}T_\alpha. $$
By the induction assumption, $X^-$ has an essentially disjoint pc-presentation. If $l=0$, then $\{ T_\alpha\}_{ \alpha \in \hat{D}}$ is an essentially disjoint pc-presentation of $\hat{X}$; we can take the disjoint combination of the former and a pc-presentation of $X^-$ to get an essentially disjoint pc-presentation of $X$. Towards using induction, assume that $l>0$ and we have proven the statement for all $X'$ with a pc-presentation $ \{ {T'}_{\alpha'}\}_{ \alpha' \in D'}$ such that for similarly defined $l'$, we have $l'<l$.

Let $ R $ be the set of $ (\alpha, \beta) \in \hat{D}_2 \times \hat{D}_2$ such that neither $  T_\alpha\sim T_\beta $ nor $T_\alpha\simperp T_\beta$. By preceding corollary $R$ is definable.
Let $ \{ T^{(1)}_{\alpha, \beta}\}_{ (\alpha, \beta) \in R}$ and  $ \{ T^{(2)}_{\alpha, \beta}\}_{ (\alpha, \beta) \in R}$ be  definable families given by 
$$T^{(1)}_{\alpha, \beta}\ =\ T_\alpha \capdot T_\beta\ \text{ and }\ T^{(2)}_{\alpha, \beta}\ =\ T_\alpha \dotminus T_\beta.$$
The above two families are definable families of pc-sets.
By the assumption of the preceding paragraph, for all $\alpha \in D_2$, there is $\beta \in D_2$, such that $( \alpha,\beta)$ is in $R$. For such $\beta$, 
$$T_\alpha\ =\ T^{(1)}_{\alpha, \beta} \cup T^{(2)}_{\alpha, \beta}.$$
Hence, taking the disjoint combination of $ \{ T_\alpha\}_{ \alpha \in \hat{D}_1}$, $ \{ T_\alpha\}_{ \alpha \in {D\backslash \hat{D}}}$, $ \{ T^{(1)}_{\alpha, \beta}\}_{ (\alpha, \beta) \in R}$, $ \{ T^{(2)}_{\alpha, \beta}\}_{ (\alpha, \beta) \in R}$, we get a new pc-presentation $ \{ T'_{\alpha'}\}_{ \alpha' \in D'}$ of $X$. We note that by (6), (7), (8), (9) of Lemma~\ref{lies}, for every $(\alpha, \beta) \in R$ and $i \in \{1,2\}$ 
$$\text{either }\ \mr_K(T^{(i)}_{\alpha, \beta})\ <\ \max_{ \beta \in \hat{D}_2} \mr_K(T_\beta) \ \text{ or } \ \md_K(T^{(i)}_{\alpha, \beta})\ <\ \max_{ \beta \in \hat{D}_2} \md_K(T_\beta).$$
Hence, for $l'$ defined for $\{ T'_{\alpha'}\}_{ \alpha' \in D'}$ in the same way as in the preceding paragraph, we have $l'<l$. The conclusion follows by the induction assumption from the preceding paragraph.
\end{proof}

\noindent
Let $\{T_\alpha\}_{\alpha \in D}$ be an essentially disjoint pc-presentation and $\hat{D}$ be its primary index set. A {\bf primary index quotient} of $\{T_\alpha\}_{\alpha \in D}$ consists of a definable subset $\widetilde{D}$ of $F^l$ for some $l$ and a map $\pi: \hat{D} \to \widetilde{D}, \alpha \mapsto \tilde{\alpha}$ such that $\tilde{\alpha}= \tilde{\beta}$ if and only if $T_\alpha \sim T_\beta$. However, we will systematically abuse notation calling $\widetilde{D}$ a primary index quotient of $T_\alpha$ and regarding the map $\pi$ as implicitly given. Moreover, we will write $\tilde{\alpha} \in \widetilde{D}$ as an abbreviation for $\alpha$ is an element of $\hat{D}$ and $\tilde{\alpha}$ is the image of $\alpha$ under $\pi$.
We note that we can always find a primary index quotient of an essentially disjoint pc-presentation by Lemma~\ref{Definability}, Lemma~\ref{QuanRed} and the fact that \text{ACF} has elimination of imaginaries. 
For a property (P) of pc-sets which is preserved under the equivalent relation $\sim$,
let $\widetilde{D}_{\text{(P)}}$ 
consist of $\tilde{\alpha} \in \widetilde{D}$ such that $T_\alpha$ satisfies (P).
With (P) as above, we say (P) holds {\bf for most} $\tilde{\alpha} \in \widetilde{D}$ if $\mr_F(\widetilde{D} \backslash \widetilde{D}_{\text{(P)}}) < \mr_F( \widetilde{D})$.

The next proposition allows us to define another geometrical invariant of a definable set based on a choice of its essentially disjoint pc-presentation and yet independent of such choice.

\begin{prop} \label{Fgeorank}
Suppose $X \subseteq X' \subseteq K^n$, $X$ has an essentially disjoint pc-presentation $\{ T_\alpha\}_{ \alpha \in D}$ with  a primary index quotient $\widetilde{D}$ and $X'$ has an essentially disjoint pc-presentation $\{ T'_{\alpha'}\}_{ \alpha' \in D'}$ with a primary index quotient $\widetilde{D}'$. Then either $\gr_K(X)< \gr_K(X')$ or $\gr_K(X)= \gr_K(X')$ and $\mr_F(\widetilde{D}) \leq \mr_F(\widetilde{D}')$.
\end{prop}

\begin{proof}
Suppose $X, X'$, $\{ T_\alpha\}_{ \alpha \in D}$, $\{ T'_{\alpha'}\}_{ \alpha' \in D'}$, $\widetilde{D}$ and $\widetilde{D}'$ are as given.
By Proposition~\ref{Kgeorank}, we have $\gr_K(X) \leq \gr_K(X')$, so suppose $\gr_K(X) = \gr_K(X')$.
Let  $\displaystyle R \subseteq \widetilde{D} \times \widetilde{D}'$ be the relation consisting of $(\tilde{\alpha}, \tilde{\alpha}') \in R$  such that
$$  \mr_K( T_{\alpha} \cap T'_{\alpha'})\ =\ \gr_K(X)\ =\ \gr_K(X')\ \text{ for } \alpha \in \pi^{-1}(\tilde{\alpha}) \text{ and } \alpha' \in \pi'{}^{-1}(\tilde{\alpha}'). $$ By (10) of Proposition~\ref{lies}, $R$ is well defined. By (1) of Corollary~\ref{Definability}  and Theorem~\ref{StablyEmbbed}, $R$ is definable in $F$.
Since $ X\subseteq X'$,  for every $ \tilde{\alpha}$ there is at least one $ \tilde{\alpha}' $ such that $  \tilde{\alpha} R \tilde{\alpha}' $.
Since the pc-presentation of $X$ is essentially disjoint for each $  \tilde{\alpha}'$, there is at most finitely many $ \tilde{\alpha}$ such that $ \tilde{\alpha} R \tilde{\alpha}'$.
Hence, $ \mr_F (\widetilde{D}) \leq  \mr_F (\widetilde{D}') $, and so the conclusion follows.
\end{proof}

\begin{cor}
If $X \subseteq K^n$ has essentially disjoint pc-presentations $\{ T_\alpha\}_{ \alpha \in D}$ and $\{ T'_{\alpha'}\}_{ \alpha' \in D'}$ with respective primary index quotients $(\widetilde{D}, \pi)$ and  $(\widetilde{D}', \pi')$, then $\mr_F(\widetilde{D})= \mr_F(\widetilde{D}')$.
\end{cor}

\noindent
Suppose $X \subseteq K^n$ has an essentially disjoint pc-presentation $\{ T_\alpha\}_{ \alpha \in D}$ with a primary index quotient $(\widetilde{D}, \pi)$. The { \bf $F$-geometric rank} of $X$, denoted by $\gr_F(X)$, is defined by $\mr_F(\widetilde{D})$. The {\bf geometric rank} of $X$, denoted by $\gr(X)$ is defined by $\omega\cdot \gr_K(X)+ \gr_F(X)$.

\begin{cor} \label{grbehavior2}
Suppose $X, X' \subseteq K^n $ are definable. Then we have the following:
\begin{enumerate}
\item if $X \subseteq X'$, then $\gr(X) \leq \gr(X')$;
\item $\gr(X \cup X') = \max\{\gr(X), \gr(X') \}$.
\end{enumerate}

\end{cor}

\begin{proof}
(1) follows from Propositions~\ref{Kgeorank} and~\ref{Fgeorank}. For (2), suppose $X, X'$ are as given. We can reduce to the case where $X, X'$ are disjoint, and $\gr(X) \geq \gr(X')$. Choose essentially disjoint pc-presentations of $X,X'$ and take the disjoint combination to get a pc-presentation of $X \cup X'$. As $X$ and $X'$ are disjoint, the obtained pc-presentation remains essentially disjoint. Calculating geometric rank yields the desired result.
\end{proof}

\noindent
We finally introduce the ultimate description of definable set in a model of $\TN_p$. Let $P$ be a system of polynomials in $ K[w, x]$ and $T$ is a subset of $K^n$. We say $P$ {\bf divides $T$} if there is some $\beta \in F^{|w|}$  such that for $W = Z\big(P( \chi(\beta),x)\big)$, we have $\mr_K(T \cap W) = \mr_K(T \backslash W)$. We note that if $T' \sim T$, and $P$ as above divides $T$ then $P$ also divides $T'$.
Let $ \{ T_\alpha\}_{ \alpha \in D}$ be an essentially disjoint pc-presentation. 
We say $P$ {\bf essentially divides} $ \{ T_\alpha\}_{ \alpha \in D}$ if  
$$\mr_F( \widetilde{D}_P)\ =\ \mr_F( \widetilde{D}) \ \text{ where }\ \widetilde{D}_P\ =\ \{ \tilde{\alpha} \in  \widetilde{D} : P \text{ divides } T_\alpha  \}  .$$ 
We note that if a system $P$ essentially divides $ \{ T_\alpha\}_{ \alpha \in D}$ then some polynomial in the system already essentially divides $ \{ T_\alpha\}_{ \alpha \in D}$.

We say $ \{ T_\alpha\}_{ \alpha \in D}$ is  a {\bf geometric pc-presentation} if  $P$ does not essentially divide $ \{ T_\alpha\}_{ \alpha \in D}$ for any $P \in K[w, x]$.

\begin{prop} \label{geoired}
Suppose $X \subseteq K^n$ is definable. The following are equivalent:
\begin{enumerate}
\item if $X' \subseteq X$ with $\gr(X')= \gr(X)$, then $\gr(X \backslash X') < \gr(X)$;
\item $X$ has a geometric pc-presentation  $ \{ T_\alpha\}_{ \alpha \in D}$ with a primary index quotient $\widetilde{D}$ satisfying $\md_F(\widetilde{D}) =1$.
\end{enumerate}
\end{prop}
\begin{proof}
Towards the proof of (1) implying (2), we make the observation that for $X' \subseteq X$ with $\gr(X \backslash X') < \gr(X)$, if (2) holds for $X'$ then (2) holds for $X$. Suppose $\{ {T'}_{\alpha'}\}_{ \alpha' \in D'}$ is a pc-presentation of $X'$ as described in (2). Choose an arbitrary essentially disjoint pc-presentation of $X \backslash X'$ and take a disjoint combination of it and $\{ {T'}_{\alpha'}\}_{ \alpha' \in D'}$ to get an essentially disjoint pc-presentation $ \{ T_\alpha\}_{ \alpha \in D}$ of $X$. The assumption that $\gr(X \backslash X') < \gr(X)$ implies that $ \{ T_\alpha\}_{ \alpha \in D}$ is still a presentation as described in (2).

We also make a few more observations and reductions. Suppose $X \subseteq K^n$ has the property (1) and $\{ T_\alpha\}_{ \alpha \in D}$ is an essentially disjoint pc-presentation of $X$ with $\hat{D}$ the primary index set and $\widetilde{D}$ a primary index quotient. It follows from (1) that $\md_F(\widetilde{D}) =1$. 
Let $l$ be the natural number such that 
$$\md_K( T_\alpha)\ =\ l\ \text{ for most } \tilde{\alpha} \in \widetilde{D}\textbf{} .$$
Suppose $l=1$, then $\{ T_\alpha\}_{ \alpha \in D}$ is geometric with the desired properties and we are done. Towards using induction, assume we have proven (2) for all $X'$ satisfying (1) with an essentially disjoint presentation  $\{ T'_{\alpha'}\}_{ \alpha' \in D'}$ such that for similarly defined $l'$, we have $l'<l$. Suppose  $\{ T_\alpha\}_{ \alpha \in D}$ is not geometric as otherwise we are done.
Then there is a polynomial $P \in K[w,x]$ such that 
$$\mr_F(\widetilde{D}_P)\ =\ \mr_F(\widetilde{D})\ \text{ where }\ \widetilde{D}_P\ =\ \{ \tilde{\alpha} \in \widetilde{D} \mid P  \text{ divides } T_\alpha \} .$$
By the observation in the preceding paragraph, we can arrange to have 
$$D\ =\ \hat{D},\  \widetilde{D}\ =\ \widetilde{D}_P\ \text{ and  for some } l,\ \md_K( T_\alpha)\ =\ l\ \text{ for all }  \tilde{\alpha} \in \widetilde{D}. $$
We prove that (1) implies (2). For each $\beta \in F^k$, let $W_\beta$ be $Z\big( P(\chi(\beta),x)\big)$. Let $D' $ be the set consisting of $(\alpha, \beta) \in D \times F^{k}$ such that  
$\mr_K( T_\alpha \cap W_\beta) = \mr_K(T_\alpha \backslash W_\beta)$. Let $X'$ be given by the pc-presentation $ \{ T_\alpha \cap W_\beta \}_{ (\alpha, \beta) \in D'}$ and $X''$ be given by the pc-presentation $ \{ T_\alpha \backslash W_\beta \}_{ (\alpha, \beta) \in D'}$. 
Then $$X\ =\ X'\cup X'',\ \text{and so }\ \gr(X)\ =\ \max\big\{\gr(X'), \gr(X'')\big\}.$$ We assume $\gr(X) = \gr(X')$; the other case can be dealt with similarly.  We note that for all $(\alpha, \beta) \in D'$, $\md_K(T_\alpha \cap W_\beta )<l$. The pc-presentation $ \{ T_\alpha \cap W_\beta \}_{ (\alpha, \beta) \in D'}$ might not be essentially disjoint, but by applying the same procedure as in Proposition~\ref{essdispre} we can produce an essentially disjoint $\{ T'_{\alpha'}\}_{ \alpha' \in D'}$. By construction, for similarly defined $l'$, we have $l'<l$. Therefore, by the assumption in the preceding paragraph $X'$ satisfies (2). By $(1)$, we have $\gr(X \backslash X')<\gr$ and so by the observation in the first paragraph $X$ also satisfies (2).

Next, we make some preparation for the proof of $(2)$ implies $(1)$. 
Towards a contradiction, suppose $X \subseteq K^n$ has a pc-presentation $\{ T_\alpha\}_{ \alpha \in D}$ with a primary index quotient $\widetilde{D}$ as in $(2)$ but there is $X'\subseteq X$ such that $\gr(X')= \gr(X \backslash X') = \gr(X)$.
Replacing $X$ by $X_0 \subseteq X$ and $\gr(X \backslash X_0)< \gr(X)$ if needed, we can arrange that  $$\mr_K(T_\alpha)\ =\ \gr_K(X)\ \text{ for all } \alpha \in D.$$ 
Choose an essentially disjoint pc-presentation $\{ {T'}_{\alpha'}\}_{ \alpha' \in D'}$ of $X'$. As usual, let $\widetilde{D}'$ be a primary index quotient of  $\{ {T'}_{\alpha'}\}_{ \alpha' \in D'}$. 
Replacing this pc-presentation by its fiber-wise intersection with $\{ T_\alpha\}_{ \alpha \in D}$ if needed,  we can arrange that for each $\alpha' \in D'$, there is $\alpha \in D$ such that ${T'}_{\alpha'} \subseteq T_\alpha$.
For each $\alpha' \in D'$, let $V'_{\alpha'}$ be the closure of ${T'}_{\alpha'}$ in the $K$-topology.
By Lemma~\ref{DefinabilityOfClosure}, the family $\{ V'_{\alpha'} \}_{\alpha' \in D'}$ is definable. By Lemma~\ref{APsimple2}, there are finitely many systems of polynomials $P'_1, \ldots, P'_k$ in $K[w', x]$ such that for all $\alpha' \in D'$, 
$$ V'_{\alpha'}\ =\ Z\Big(P_i'\big(\chi(\beta'),x\big)\Big)\ \text{ for some } \beta' \in F^{|w'|} \text{ and } i \in \{1, \ldots, k\}.$$ 
By shrinking $X'$ if needed, we can assume that $\md_F(\widetilde{D}') =1$. 
By shrinking $X'$ further, we can assume there is  a system $P'$ in $K[w', x]$ such that for any $\alpha' \in D'$, there is $\beta'$ in $F^{|w'|}$ with $ V'_{\alpha'} =Z\big(P'(\chi(\beta'),x)\big)$. 
Let $\{ {T''}_{\alpha''}\}_{ \alpha'' \in D''}$ be an essentially disjoint pc-presentation of $X \backslash X'$. Replacing this pc-presentation by its fiber-wise intersection with $\{ T_\alpha\}_{ \alpha \in D}$ if needed,  we arrange that  for all $\alpha'' \in D''$, there is $\alpha \in D$ such that ${T''}_{\alpha''} \subseteq T_\alpha$.

We continue the proof of $(2)$ implies $(1)$ with the notations as in the preceding paragraph. 
Set
$$\widetilde{D}_1\  =\ \{ \tilde{\alpha} \in \widetilde{D} : \text{ there is } \tilde{\alpha}' \in \widetilde{D}' \text{ such that } T'_{\alpha'}\subsim T_\alpha \}\ \text{ and }\ X_1\ =\ \bigcup_{\tilde{\alpha} \in \widetilde{D}_1} T_\alpha. $$
$$\widetilde{D}_2\  =\ \{ \tilde{\alpha} \in \widetilde{D} : \text{ there is } \tilde{\alpha}'' \in \widetilde{D}'' \text{ such that } T''_{\alpha''}\subsim T_\alpha \}\ \text{ and }\ X_2\ =\ \bigcup_{\tilde{\alpha} \in \widetilde{D}_2} T_\alpha. $$
By the arrangement in the preceding paragraph, for each $\tilde{\alpha}' \in \widetilde{D}'$, there is $\tilde{\alpha} \in \widetilde{D}$ such that ${T'}_{\alpha'} \subsim T_\alpha$, so $\gr( X' \backslash X_1) < \gr(X)$. 
 As $X_1 \subseteq X $, $\gr(X_1) \leq \gr(X)$. By assumption, $\gr(X) =\gr(X')$. By putting the last three statements together and using Corollary~\ref{grbehavior2}, we have 
$$\gr( X')\ =\ \gr(X_1)\ =\ \gr(X)\ \text{ and so }\ \mr_F(\widetilde{D}_1)\ =\ \mr_F(\widetilde{D}')\ =\ \mr_F(\widetilde{D}).$$
Similarly, we get $\mr_F(\widetilde{D}_2) =\mr_F(\widetilde{D}'')=\mr_F(\widetilde{D}) $.
As  $\md_F(\widetilde{D})=1$, $\mr_F(\widetilde{D}_1 \cap \widetilde{D}_2) = \mr_F(\widetilde{D}) $. 
We will now show that 
$$ P' \text{ divides }T_\alpha \text{ for all } \tilde{\alpha} \in \widetilde{D}_1 \cap \widetilde{D}_2. $$
This is the desired contradiction as $P'$ will then essentially divide $\{ T_\alpha\}_{ \alpha \in D}$. Suppose  $\tilde{\alpha} $ is in $\widetilde{D}_1 \cap \widetilde{D}_2$.  
Then there are $\tilde{\alpha}' \in \tilde{D}'$ and $\tilde{\alpha}'' \in \tilde{D}''$ such that ${T'}_{\alpha'} \subsim T_\alpha$ and ${T''}_{\alpha''} \subsim T_\alpha$.
 By preceding paragraph, there is $\beta' \in F^{|w'|}$, such that the closure $V'_{\alpha'}$  of $T'_{\alpha'}$ in $K$-topology is $Z\big(P'(\chi(\beta'),x)\big)$. 
Then $$\mr_K(T_\alpha \cap V'_{\alpha'})\ \geq\ \mr_K({T'}_{\alpha'})\ =\ \mr_K(T_{\alpha}),\ \text{ and so }\ \mr_K(T_\alpha \cap V'_{\alpha'})\ =\ \mr_K(T_{\alpha}).$$ 
On the other hand, by $(10)$ of Proposition~\ref{lies} and the fact that ${T'}_{\alpha'} \cap {T''}_{\alpha''}$ is empty, $\mr_K({T''}_{\alpha''} \cap V'_{\alpha'})< \mr_K(T_{\alpha})$. By $(7)$ of Proposition~\ref{lies}, $\mr_K({T''}_{\alpha''} \backslash V'_{\alpha'}) = \mr_K(T_{\alpha})$. 
Therefore, $$\mr_K(T_\alpha \backslash V'_{\alpha'})\ \geq\ \mr_K({T''}_{\alpha''} \backslash V'_{\alpha'})\ =\ \mr_K(T_{\alpha}),\ \text{ and so }\ \mr_K({T}_{\alpha} \backslash V'_{\alpha'})\ =\ \mr_K( T_\alpha).$$
Thus,  $\mr_K(T_\alpha \cap V'_{\alpha'}) = \mr_K({T}_{\alpha} \backslash V'_{\alpha'})$ and $P'$ divides $T_\alpha$ as desired.
\end{proof}

\noindent
If $X \subseteq K^n$ is definable and satisfies one of the two statements of Proposition~\ref{geoired}, we say $X$ is {\bf geometrically irreducible}. 

\begin{lem} \label{irredlemma}
Suppose $X, X', X'' \subseteq K^n$ are definable. Then
\begin{enumerate}
\item if $X$ is geometrically irreducible and $\gr(X')< \gr(X)$, then $X \cup X'$ is geometrically irreducible;
\item if $X$ is geometrically irreducible and $\gr(X')< \gr(X)$, then $X \backslash X'$ is geometrically irreducible;
\item if $X, X', X''$ are geometrically irreducible, $\gr(X \cap X')=\gr(X)=\gr(X')$ and $\gr(X'\cap X'') = \gr(X') = \gr(X'')$, then $\gr(X \cap X'') = \gr(X) =\gr(X'')$.
\end{enumerate}
\end{lem}
\begin{proof}
We have that $(1)$ follows from condition $(2)$ of Proposition~\ref{geoired} and $(2)$, $(3)$ follow from condition $(1)$ of Proposition~\ref{geoired}.
\end{proof}

\begin{thm}
Every definable $X\subseteq K^n$ has a geometric pc-presentation. 
\end{thm}

\begin{proof}
We first prove an auxiliary result. Suppose $X$ has an essentially disjoint pc-presentation $\{ T_\alpha\}_{ \alpha \in D}$  with a primary index  quotient $\widetilde{D}$ satisfying $\md_F( \widetilde{D}) =1$ and there is $l>0$ such that $$\mr_K(T_\alpha)\ =\ \gr_K(X)\ \text{ and }\ \md_K(T_\alpha)\ =\ l\ \text{ for all } \tilde{\alpha} \in \widetilde{D}. $$
We will show that $X$ can be written as a disjoint union of at most $l$  definable sets, all with geometrical rank $\gr(X)$. Suppose towards a contradiction that $X$ can be written as a disjoint union of $ X_1, \ldots, X_{l+1}$, all with of geometrical rank $\gr(X)$. For each $i \in \{1, \ldots, l+1\}$, let $\{ T_{i,\beta}\}_{ \beta \in E_i}$ be an essentially disjoint pc-presentation of $X_i$, set 
$$D_i\ =\ \big\{ \alpha \in D : \text{ there is } \beta \in E_i \text{ with  }\mr_K( T_\alpha \cap T_{i,\beta}) = \gr_K(X) \big\}$$  and let $\widetilde{D}_i$ be a primary index quotient of $\{ T_\alpha\}_{ \alpha \in D_i}$. As $\gr(X_i) = \gr(X)$, we have $\mr_F(\widetilde{D}_i)=\mr_F(\widetilde{D})$ for $i \in \{1, \ldots, l+1\}$. Hence, we can find $\alpha $ in $\bigcap_{i =1}^{l+1} D_i$. Then $\md_K(T_\alpha)> l$, a contradiction.

Suppose $X \subseteq K^n$ is definable. We will next show there is a finite collection of geometrically irreducible definable sets $\{X_i\}_{i \in I}$ each with $\gr(X_i)=\gr(X)$ such that $X = \bigcup_{i \in I} X_i$.
If $X = X' \cup X''$ such that $\gr(X')=\gr(X'') =\gr(X)$, and we have proven the statement for $X', X''$, the statement for $X$ also follows.
Therefore, we can arrange that $X$ has an essentially disjoint pc-presentation $\{ T_\alpha\}_{ \alpha \in D}$ where $\md_F( \widetilde{D}) =1$. 
By possibly removing a set of geometrical rank smaller than $\gr(X)$ and using $(1)$ of the preceding lemma, we can also arrange that there is $l>0$ such that for all $\alpha \in D$,  $\mr_K(T_\alpha) =\gr_K(X)$ and $\md_K(T_\alpha)=l$. By the preceding paragraph, $X$ can be written as a disjoint union of at most $l$  definable sets, all of geometrical rank $\gr(X)$. Let $X$ be written as a disjoint union of the most number of definable sets, all of geometrical rank $\gr(X)$. Then it can be easily seen that each of these must be geometrically irreducible.

We now prove the proposition. Suppose $X = \bigcup_{i\in I} X_i$ is as in the previous paragraph. For $i, j \in I$, we write $X_i \approx X_j$ if $\gr(X_i \cap X_j) =\gr(X)$. By $(2)$ of the preceding lemma, $\approx$ is an equivalence relation on $\{X_i\}_{i \in I}$. We can choose $J \subseteq I$ such that $\{X_j\}_{j \in J}$ are representatives of the equivalent classes. Then using Corollary~\ref{grbehavior2} and characterization $(1)$ of geometric irreducibility, it is easy to see that $X =X' \cup \left( \bigcup_{j \in J} X_j\right)$ such that $\gr(X')<\gr(X)$. Using $(1)$ of the preceding lemma, we can reduce to the case where $X =\bigcup_{j \in J} X_j$. By $(2)$ of the preceding lemma, we can further reduce to the case where $X_i, X_j$ are disjoint for distinct $i, j \in J$. For all $i \in J$, $X_i$ has a geometric pc-presentation.  Then $X$ has a geometric  pc-presentation which is a combination of the geometric pc-presentations of $X_i$.
\end{proof}

\noindent With the next proposition, we get the final geometric invariant of a definable set which is based on a choice of its geometric presentation but independent of such choice.

\begin{prop}
There is a unique $d$ such that $X$ is a disjoint union of $d$ geometrically irreducible sets of geometric rank $\gr(X)$. Moreover, if $\{ T_\alpha\}_{ \alpha \in D}$ is a geometric pc-presentations of $X \subseteq K^n$ then $\md_F(\widetilde{D})=d$.
\end{prop}
\begin{proof}
Suppose $\{ T_\alpha\}_{ \alpha \in D}$ is a geometric pc-presentation of $X$ with $\hat{D}$ its primary index set, $\widetilde{D}$ its primary index quotient and $\md_F(\widetilde{D})=d$. We will show that $X$  is a disjoint union of $d$ geometrically irreducible sets of geometric rank $\gr(X)$.
Let $\widetilde{D}_1, \ldots, \widetilde{D}_d$ be the disjoint subsets of $\widetilde{D}$ such that $\widetilde{D} = \bigcup_{i=1}^d \widetilde{D}_i$ and  $\mr_F(\widetilde{D}_i)=1$  for each $i \in \{1, \ldots,d\}$. Let $\hat{D}_1, \ldots, \hat{D}_d \subseteq \hat{D}$ be the corresponding inverse images under the canonical map $\hat{D} \to \widetilde{D}$.
Set $$X_1\ =\ \bigcup_{\alpha \in \hat{D}_1 \cup D\backslash \hat{D}} T_\alpha\ \text{ and }\ X_i\ =\ \bigcup_{\alpha \in D_i} T_\alpha \  \text{ for }\ i \in \{2, \ldots, d \}.$$ Since $\{ T_\alpha\}_{ \alpha \in D}$ is geometric, each $X_i$ is geometrically irreducible of geometrical rank $\gr(X)$.

It remains to show the uniqueness part of the first statement of the proposition. Suppose the above $X$ is the disjoint union of geometrically irreducible sets $X'_1, \ldots, X'_{d'}$, all of geometrical rank $\gr(X)$. Then for every $i \in \{1, \ldots, d\}$, there is a unique $j \in \{1, \ldots, d'\}$ such that $\gr(X_i \cap X'_{j}) = \gr(X_i)$ . Conversely, for each $j \in \{1, \ldots, d'\} $, there is a unique $i \in \{1, \ldots, d\}$ such that $\gr(X_i \cap X'_{j}) = \gr(X_i) $. The conclusion follows.
\end{proof}

\noindent
For each $X \subseteq K^n$, we call the number $d$ as in the above proposition the {\bf geometric degree} of $X$ and denote this by $\gd(X)$. The following properties of this notion are immediate from the preceding proposition.

\begin{cor} \label{gdbehavior}
Suppose $X, X'$ are definable and $\gr(X)=\gr(X')$. We have the following:
\begin{enumerate}
\item if $X \subseteq X'$, then $\gd(X) \leq \gd(X')$;
\item if $\gr(X \cap X')< \gr(X)$, then $\gd(X \cup X') =\gd(X) +\gd(X')$;
\item if $\gr(X \cap X')= \gr(X)$, then $ \gd(X \cup X') =\gd(X) +\gd(X') - \gd(X \cap X')$.
\end{enumerate}
\end{cor}

%% file: 7.OmegaStabilityAndMorleyRank.tex
\section{Rank, degree and their behaviors}

\noindent
We continue working in a fixed model \( (F, K; \chi) \) of \( \TN_p\) and keeping the notations and conventions of the preceding section. We show in this section that the geometric rank and geometric degree defined in the preceding section agrees with Morley rank and Morley degree.

It is also convenient to define the so-called Cantor rank and Cantor  degree for a Boolean algebra. Suppose $B$ is a Boolean algebra and $b \in B$, we set $\crk(b) = -\infty$ for $b = 0$ and $\crk(b) \geq 0$ if $b \neq 0$. For an ordinal $\alpha$, we set $\crk(b) \geq \alpha $ if for all $\beta< \alpha$ there is an infinite disjoint family $\{ b_n \}_{n \in \nn}$ with $\crk(b_i) \geq \beta$ such that $b\wedge b_i = b_i$. Set 
$$\crk(b)\ =\ \max \big\{ \alpha : \crk(b) \geq \alpha \big\}\ \text{ if  this exists and }\ \crk(b)\ =\ \infty\ \text{ otherwise}.$$
If $b = b_1 \vee \cdots \vee b_m$ with $ b_1, \ldots, b_m$ disjoint and $ \crk(b_i) = \crk(b) < \infty $ for $i \in \{ 1, \ldots, m \}$, we say $\cdg(b) \geq m$. Set 
$$\cdg(b)\ =\ \max \big\{ m : \cdg(b) \geq m \big\}.$$
We apply the above definition in the case where $B$ is the Boolean algebra of definable subsets of $F^m$ or of $K^n$. Note that if $(F, K; \chi) $ is $\aleph_0$-saturated, then $\crk(X)= \mmr(X)$ and $\cdg(X)= \mmd(X)$ for all definable $X \subseteq K^n$.

\begin{prop}
If $D \subseteq F^m$ is definable, then $\mmr(D)= \crk(D)=\mr_F(D)$ and $\mmd(D)=\cdg(D)=\md_F(D)$.
\end{prop}
\begin{proof}
By Theorem~\ref{StablyEmbbed} if $D \subseteq F^m$ is definable, then $D$ is already definable in $F$ as a field. The conclusion follows.
\end{proof}

\noindent Let $P$ be an $m$-ary second-order property about definable sets in models of $\TN_p$ and $X_1, \ldots, X_m$ be definable in $(F, K;\chi)$.  We say $P(X_1, \ldots, X_m)$ is {\bf preserved under elementary extensions} if for every elementary extension $(F', K';\chi')$ of $(F, K;\chi)$, we have that $P(X_1, \ldots, X_m)$ is equivalent to $P(X_1', \ldots, X'_m)$ where $X'_i$ is defined by the $L$-formula with parameters in $K$ defining $X_i$ for each $i \in \{1, \ldots, m\}$. We define being {\bf preserved under elementary extensions} likewise for $f(X_1, \ldots, X_m)$ where $f$ is an $m$-ary function on definable sets in models of $\TN_p$. 

\begin{lem} \label{Preservative}
If $X\subseteq K^n$ is definable, then $ \gr(X)$ and $\gd(X)$ are preserved under elementary extensions.
\end{lem}

\begin{proof}
We first note that if $T, V \subseteq K^n$, then the property that $V$ is the closure of $T$ in the $K$-topology is preserved under elementary extensions. Suppose $X \subseteq K^n$ has a geometric pc-presentation $\{T_\alpha \}_{\alpha \in D}$ with primary index quotient $\widetilde{D}$. This fact is preserved under elementary extensions as the notion of geometric pc-presentation is defined using notions of closure in $K$-topology, $K$-Morley rank and degree, $F$-Morley rank and degree. Finally, $\gr(X)$ and $\gd(X)$ are calculated using $\{T_\alpha \}_{\alpha \in D}$, $\widetilde{D}$, $K$-Morley rank, $F$-Morley rank and $F$-Morley degree which are invariant under elementary extensions. Hence, the conclusion follows.
\end{proof}

\begin{lem}
If $X \subseteq K^n$ is definable,
then $\gr(X) = \crk(X)$ and $\gd(X)=\cdg(X)$. 
\end{lem}

\begin{proof}
Suppose $X \subseteq K^n$ is definable. Let $\{ T_\alpha \}_{\alpha \in D }$ be a geometric pc-presentation of $X$ with primary index quotient $\widetilde{D}$. 

First, consider the case when $\gr(X)=0$. Using Corollary~\ref{grbehavior2}, we can reduce to the case when $\gd(X)=1$.  
Then  $|\widetilde{D}| = 1$ and for each $\tilde{\alpha} \in \widetilde{D}$, $T_\alpha$ is finite. 
It follows from the latter that  if $T_\alpha \sim T_\beta$ then $T_\alpha =T_\beta$ for $\tilde{\alpha}, \tilde{\beta} \in \widetilde{D}$. 
Therefore, for each $\tilde{\alpha} \in \widetilde{D}$, no $P \in K[w, x]$ divides $T_\alpha$. Hence, for each $\tilde{\alpha} \in \widetilde{D}$, $T_\alpha$ has only one element. Thus $|X| = 1$, $\crk(X)= \gr(X) =0$ and $\cdg(X)=\gd(X)=1$.

Towards a proof by induction, suppose we have shown the statement for all $Y$ with $\gr(Y)< \gr(X)$. We will next show that $\crk(X)\geq \gr(X)$.
Again by using Corollary~\ref{grbehavior2}, we can reduce to the case when $\gd(X)=1$.
First, consider the case when $\mr_F(\widetilde{D})=\gr_F(X)>0$.
We can choose disjoint family $\{\widetilde{D}_i \}_{i \in \nn}$ of $F$-definable subsets of $\widetilde{D}$ such that 
$$\mr_F(\widetilde{D}_i)\ =\ \mr_F(\widetilde{D}) - 1\ \text{ for }\ i\in \nn.$$
Let $X_i = \bigcup_{\alpha \in D_i} T_\alpha$ for $i\in \nn$.
Then $\gr(X_i)= \gr(X)-1 \text{ for all } i\in \nn$. Moreover, $$\gr_K(X_i \cap X_j)\ <\ \gr_K(X)\ \text{ and so }\ \gr(X_i \cap X_j)\ <\ \gr(X)-1\ \text{ for distinct }\ i,j \in \nn .$$
By induction hypothesis, $\crk(X_i) = \gr(X)-1 $ for all $i \in \nn$ and $\crk(X_i \cap X_j)<\gr(X)-1$ for distinct $i, j \in \nn$.
Thus, $\crk(X) \geq \gr(X)$.

We continue showing that $\crk(X)\geq \gr(X)$ for the remaining case when $\gd(X)=1$, $\mr_F(\widetilde{D})=\gr_F(X)=0$. 
Then $\widetilde{D} = \{ \tilde{\alpha} \}$ and no $P \in K[w, x]$ divides $T_\alpha$. Using the induction hypothesis, we can reduce to the case when $X = T_\alpha$ is an irreducible variety of dimension $\gr_K(X)$.
By Noether normalization lemma, we can further reduce to the case when $$X\ =\ K^m\ \text{ with }\ m=\gr_K(X).$$
We have previously covered the case when $m = 0$. If $m>0$, let $a_1, \ldots, a_k \in K$ be algebraically independent elements over $\qq\big(\chi(F)\big)$. Set $$Y\ =\ a_{1}\chi(F) + \cdots  + a_{k} \chi(F)\ \text{ and }\ Z\ =\ Y \times K^{m-1}.$$
Then $Z\subseteq X$ is in a one-to-one correspondence with $\chi(F)^k \times K^{m-1}$.
Note that $\chi(F)^k \times K^{m-1}$ has a geometric pc-presentation $\{ T'_{\beta} \}_{ \beta \in F^k}$ where for each $\beta \in F^k$, $T'_\beta = \big\{ \chi(\beta) \big\} \times K^{m-1}$. Hence, by induction hypothesis, $\crk(Z)= \omega \cdot (m-1) +k$. Thus, $\crk(X) \geq \omega\cdot m = \gr(X)$.

Now we will prove that  $\crk(X) \leq \gr(X)$. By Corollary~\ref{grbehavior2}, we can reduce to the case when $\gd(X)=1$. Using the induction hypothesis, it suffices to prove that for any partition of $X$ into a union of two disjoint definable sets, one of them has Cantor rank less than $\gr(X)$. This follows from Corollary~\ref{gdbehavior}. 

Finally, we will verify that $\cdg(X)=\gd(X)$. Using Corollary~\ref{gdbehavior}, we can reduce to the case when $\gd(X)=1$. Suppose $\cdg(X)>1$, then $X$ is the disjoint union of $X_1, X_2$, where $\gr(X_1) =\crk(X_1)=\crk(X_2) =\gr(X_2)$. But by Corollary~\ref{gdbehavior} again, we get $\gd(X)>1$, a contradiction.
\end{proof}

\begin{thm} \label{rankAgreement}
Suppose $X \subseteq K^n$ is definable. Then $\gr(X) = \crk(X) =\mmr(X)$ and $\gd(X)=\cdg(X)=\mmd(X)$. 
\end{thm}

\begin{proof}
This follows from the preceding two lemmas and the fact that in an $\aleph_0$-saturated model of $\TN_p$, Cantor rank agrees with Morley rank and Cantor degree agrees with Morley degree.
\end{proof}

\begin{cor}
The theory $\TN$ is $\omega$-stable. 
\end{cor}

\begin{proof}
For any definable $X \subseteq K^n$, $\mmr(X) = \gr(X) < \omega^2$. The conclusion follows.
\end{proof}
\newpage
\noindent
Another consequence of Theorem~\ref{rankAgreement} is the following:

\begin{prop}
For any strongly minimal set $X \subseteq K^n$, there is a finite-to-one definable map from $X$ to $F$.
\end{prop}
\begin{proof}

We first show that if $Y \subseteq F^k$ is a strongly minimal set, then there is a finite-to-one definable map from $Y$ to $F$.
Indeed, by Theorem~\ref{StablyEmbbed}, the algebraically closed field $F$ is stably embedded into $(F, K; \chi)$ and hence $Y$ is strongly minimal in the field $F$.
By quantifier elimination of ACF, $Y$ differs from an irreducible variety $Y'$ over the field $F$ by finitely many elements.
By Noether normalization lemma, there is a finite map from $Y'$ to $F$. This can be easily modified to give a definable finite-to-one map from $Y$ to $F$.

Suppose $X$ is as in the statement of the proposition.
Then $\gr(X)=\gd(X)=1$.
Let $\{ T_\alpha\}_{\alpha \in D}$ be a geometric pc-presentation of $X$ with a primary index quotient $\widetilde{D}$.
Then for each $\tilde{\alpha} \in \widetilde{D}$, $T_\alpha$ has finitely many elements. Moreover, $T_\alpha$ and $T_\beta$  either consist of the same elements or are disjoint for all $\tilde{\alpha},\tilde{\beta} \in \widetilde{D}$.
We have a finite-to-one map from $X$ to $\widetilde{D}$ given by $a \mapsto \tilde{\alpha}$ if $a$ is an element of $T_\alpha$. By the previous paragraph, there is a finite-to-one definable map from $\widetilde{D}$ to $F$. The desired map is the composition of these two maps. 
\end{proof}

\noindent
We will next prove that Morley rank is definable in families in a model of $\TN_p$.
An {\bf fpc-formula} is an $L$-formula of the form 
$$P(x, z) \wedge \neg \exists s \Big( \phi(s, z) \wedge Q\big(\chi(s),x, z\big)\Big)$$
where $P$ is a system of polynomials in $\qq[x, z]$, $\phi(s, z)$ is a parameter free $L$-formula and $Q$ a system of polynomials in $\qq[w,x,z]$. If  $\psi(x,z)$  is an fpc-formula and $T_c \subseteq K^n$ is the set defined by $\psi(x, c)$ for $c \in K^l$, then $\{ T_c\}_{c \in K^l}$ is a definable family of pc-sets. On the other hand, as a consequence of Lemma~\ref{APsimple1}, an arbitrary pc-set is defined by $\psi(x,c)$ for some fpc-formula $\psi(x,z)$ and some $c$ in $K^{l}$.

\begin{lem}  \label{APsimple3}
Suppose \( X \subseteq K^n\) has an algebraic presentation $\{T_\alpha\}_{ \alpha \in D}$. Then, there are fpc-formulas $\psi_1(x,z), \ldots, \psi_k(x,z)$ such that for each $\alpha \in D$, there are $c \in K$ and $i \in \{1, \ldots, k\}$ such that $T_\alpha$ is defined by $\psi_i( x,c)$.
\end{lem}
\begin{proof}
For each $\alpha$ in $D$, there is a formula $\psi(x,z)$ such that for some $c \in K$, $T_\alpha$ is defined by $\psi(x,c)$. There are only countably many fpc-formulas. The conclusion follows by a standard compactness argument.
\end{proof}

\begin{thm}
Suppose $ \{ X_b\}_{b \in Y}$ is a definable family of subsets of $K^n$. Then for each ordinal \(\rho\), the set \( \big\{ b \in Y :  \gr(X_b)= \rho\big\} \) is definable.
\end{thm}

\begin{proof}
Suppose $ \{ X_b\}_{b \in Y}$ is as stated. Let $L_{\mathrm{r}}$ be the language of rings. For each choice $\mathscr{C}$ of an $L$-formula $\eta(u,x,z)$, an $L$-formula $\delta(u,z)$, an $L$-formula $\hat{\delta}(u,z)$,   an $L_{\mathrm{r}}$-formula $\tilde{\delta}(u',s)$,  an $L$-formula $\phi(u,u',z)$, fpc-formulas $\psi_1(x,z'), \ldots, \psi_e(x,z')$ (with $e \in \nn^{\geq 1}$), we will define the relation 
$$R_{\mathscr{C}}\ \subseteq\ F^{m} \times  F^{m}\times F^{k} \times K^n \times Y \times K^{l} .$$
The relation $R_{\mathscr{C}}(\alpha, \beta, \gamma, a, b, c)$ holds if and only if for  $H_\alpha$ defined by $\eta(\alpha,x,c)$, $D$ defined by ${\delta}(u,c)$, $\hat{D}$ defined by $\hat{\delta}(u,c)$,  $\widetilde{D}$ defined by $\tilde{\delta}(u',\gamma)$, $\pi$ defined by $\phi(u,u',c)$, and  $T_{i,c'}$ defined by $\psi_i( x,c')$, we have the following:
\begin{enumerate}[(a)]
\item $X_b =\bigcup_{\alpha \in D} H_\alpha$;
\item for each  $\alpha \in D$,  there are $c' \in K^{l'}$ and $i \in \{1, \ldots, e \} $ such that $H_\alpha = T_{i,c'}$;
\item $\alpha$ is in $\hat{D}$ if and only if for all $\beta \in D$,  $c' \in K^{l'}$, $i \in \{1, \ldots, e \} $, $d' \in K^{l'}$, $j \in \{1, \ldots, e\} $ with  $H_\alpha = T_{i,c'}$, $H_\beta = T_{j,d'}$, we have $\mr_K(T_{i,c'}) \geq \mr_K(T_{j,d'})$;
\item if $\alpha, \beta$ are in $\hat{D}$, for all $\beta \in D$,  $c' \in K^{l'}$, $i \in \{1, \ldots, e \} $, $d' \in K^{l'}$, $j \in \{1, \ldots, e \} $ with  $H_\alpha = T_{i,c'}$, $H_\beta = T_{j,d'}$, then we have either $T_{i,c'} \sim T_{j,d'}$ or $T_{i,c'} \simperp T_{j,d'}$;
\item $\pi$ is a surjective function from $\hat{D}$ to $\widetilde{D}$; $\pi(\alpha) =\pi(\beta)$ if and only if for all $c' \in K^{l'}$, $i \in \{1, \ldots, e \} $, $d' \in K^{l'}$, $j \in \{1, \ldots, e \} $ with  $H_\alpha = T_{i,c'}$, $H_\beta = T_{j,d'}$, we have $T_{i,c'} \sim T_{j,d'}$.
\end{enumerate}
We note that (a) and (b) can clearly be translated into statements involving $\eta(u,x,z)$, $\delta(u,z)$, $\hat{\delta}(u,z)$  $\tilde{\delta}(u',s)$, $\phi(u,u',z)$ and $\psi_1(x,z'), \ldots, \psi_e(x,z')$. For (c), (d) and (e), we can do so with the further use of Corollary~\ref{Definability}. Hence, $R_{\mathscr{C}}$ is definable. Let $R^5_{\mathscr{C}}$ be the projection of $R_{\mathscr{C}}$ on $Y$. Then $R^5_{\mathscr{C}}$ is also definable. Intuitively, $R^5_{\mathscr{C}}(b)$ means $X_b$ has an essentially disjoint pc-presentation given by $\mathscr{C}$ with some parameters.

For each $b \in Y$, there is one choice  $\mathscr{C}$ as above such that $R^5_{\mathscr{C}}(b)$.  Also, there are only countably many such choices $\mathscr{C}$. By a standard compactness argument,  $Y$ is covered by the union of $R^5_{\mathscr{C}}$ for finitely many such choices $\mathscr{C}$. We can reduce the problem to the case where $Y$ is covered by $R^5_{\mathscr{C}}$ for one choice $\mathscr{C}$ as above. 
Suppose an ordinal $\rho = \omega\cdot \rho_{{}_K} + \rho_{{}_F}$ is given where $\rho_{{}_K}, \rho_{{}_F}$ are natural number. With the notation as in the definition of $R_{\mathscr{C}}$, we have $\gr(X_b) = \rho$ if and only if $r_F(\widetilde{D}) = \rho_{{}_F} $ and $r_K(T_{i,c'}) = \rho_{{}_K}$ under the condition $T_{i,c'} = H_\alpha $ for some $\alpha \in \hat{D}$. The former is definable by the fact that Morley rank is definable in families in a model of ACF and the latter is definable as a consequence of Corollary~\ref{Definability}.
\end{proof}

\begin{cor}
Suppose $ \{ X_b\}_{b \in Y}$ is a definable family of subsets of $K^n$. Then there are ordinals $\rho_1, \ldots, \rho_k$ such that for all \(b\), for some $i \in \{1, \ldots, k\}$, $\gr(X_b) =\rho_i$.
\end{cor}

\begin{proof}
This is a consequence of the preceding theorem Lemma~\ref{Preservative} and a standard application of compactness.
\end{proof}

\begin{prop}
If $X \subseteq K^n$ and $ X'\subseteq K^{n'}$ are definable, $\gr(X) = \omega\cdot \rho_{{}_K} + \rho_{{}_F}$ and $  \gr(X') = \omega\cdot \rho'_{{}_K} + \rho'_{{}_F}$ with $\rho_{{}_K}, \rho_{{}_F}, \rho'_{{}_K}, \rho'_{{}_F} \in \nn$,  then  $\gr(X\times X') = \omega\cdot (\rho_{{}_K}+\rho'_{{}_K}) +\rho_{{}_F}+\rho'_{{}_F}$.
\end{prop}

\begin{proof}

We give a proof by induction on geometric rank of $X$ and $X'$. The case where $\gr(X)=\gr(X')=0$ is clear. Suppose we have proven the proposition for all $Y, Y'$ such that $\gr(Y)\leq \gr(X)$ and $\gr(Y') \leq \gr(X')$ and at least one of the equalities does not hold.  Let $\{T_\alpha\}_{ \alpha \in D}$ be an essentially disjoint pc-presentation of $X$ with primary index quotient $\widetilde{D}$. Likewise, we define $\{T'_{\alpha'}\}_{ \alpha' \in D'}$ and $\widetilde{D}'$. Then $\{T_\alpha \times T'_{\alpha'}\}_{ (\alpha, \alpha') \in D\times D'}$ is a fiberwise product of $\{T_\alpha\}_{ \alpha \in D}$ and  $\{T'_{\alpha'}\}_{ \alpha' \in D'}$, and hence an essentially disjoint pc-presentation of $X \times X'$. We note that, by $(11)$ of Lemma~\ref{lies}, we have 
$$\max_{(\alpha, \alpha') \in D\times D'}\gr_K(T_\alpha \times T'_{\alpha'})\ =\ \max_{\alpha\in D} \gr_K(T_\alpha) +  \max_{\alpha' \in D'} \gr_K(T'_{\alpha'}).$$
Also, by part (12) of Lemma~\ref{lies}, a primary index quotient of $\{T_\alpha \times T'_{\alpha'}\}_{ (\alpha, \alpha') \in D\times D'}$ can be chosen to be $\widetilde{D} \times \widetilde{D}'$. The desired conclusion follows.
\end{proof}

\section*{Acknowledgments}
\noindent
The authors thank Lou van den Dries for numerous discussions and comments on this paper.